\documentclass[reqno,11pt]{amsart}

\usepackage[utf8]{inputenc}
\usepackage[T1]{fontenc}

\usepackage[letterpaper, margin=1in]{geometry}
\usepackage{microtype}

\usepackage{amsmath, amssymb, amsthm, bm, mathtools, enumitem, caption, tikz}

\usepackage{hyperref}

\usepackage[noabbrev,capitalize]{cleveref}
\crefname{equation}{}{}
\Crefname{conjecture}{Conjecture}{Conjectures}

\newtheorem{theorem}{Theorem}[section]
\newtheorem{lemma}[theorem]{Lemma}
\newtheorem{corollary}[theorem]{Corollary}
\newtheorem{proposition}[theorem]{Proposition}
\newtheorem{conjecture}[theorem]{Conjecture}
\newtheorem{problem}[theorem]{Problem}

\theoremstyle{definition}
\newtheorem{definition}[theorem]{Definition}
\newtheorem{example}[theorem]{Example}

\theoremstyle{remark}
\newtheorem*{remark}{Remark}
\newtheorem{claim}{Claim}

\newenvironment{claimproof}[1][Proof]{\begin{proof}[#1]}{\end{proof}}

\DeclarePairedDelimiter\abs{\lvert}{\rvert}
\DeclarePairedDelimiter\floor{\lfloor}{\rfloor}

\DeclarePairedDelimiter\set{\{}{\}}

\newcommand{\R}{\mathbb R}
\newcommand{\N}{\mathbb N}
\newcommand{\T}{\intercal}

\DeclareMathOperator{\rnk}{rank}
\DeclareMathOperator{\tr}{tr}
\DeclareMathOperator{\rank}{rank}
\DeclareMathOperator{\mult}{mult}

\newcommand{\eps}{\varepsilon}
\newcommand{\wt}[1]{\widetilde{#1}}

\title[Spherical two-distance sets]{Spherical two-distance sets \\ and eigenvalues of signed graphs}
\author[Jiang]{Zilin Jiang}
\author[Tidor]{Jonathan Tidor}
\author[Yao]{Yuan Yao}
\author[Zhang]{Shengtong Zhang}
\author[Zhao]{Yufei Zhao}
\thanks{
Jiang was supported by an AMS Simons Travel Grant and NSF Award DMS-1953946.
Tidor was supported by the NSF Graduate Research Fellowship Program DGE-1745302.
Yao and Zhang were supported by MIT UROP.
Zhao was supported by NSF Award DMS-1764176, the MIT Solomon Buchsbaum Fund, and a Sloan Research Fellowship.
}

\address{Jiang: Arizona State University, Tempe, AZ, USA}
\email{zilinj@asu.edu}

\address{Tidor, Yao, Zhang, Zhao: Massachusetts Institute of Technology, Cambridge, MA, USA}
\email{\{jtidor,yyao1,stzh1555,yufeiz\}@mit.edu}

\begin{document}

\begin{abstract}
  We study the problem of determining the maximum size of a spherical two-distance set with two fixed angles (one acute and one obtuse) in high dimensions. Let $N_{\alpha,\beta}(d)$ denote the maximum number of unit vectors in $\mathbb R^d$ where all pairwise inner products lie in $\{\alpha,\beta\}$. For fixed $-1\leq\beta<0\leq\alpha<1$, we propose a conjecture for the limit of $N_{\alpha,\beta}(d)/d$ as $d \to \infty$ in terms of eigenvalue multiplicities of signed graphs. We determine this limit when $\alpha+2\beta<0$ or $(1-\alpha)/(\alpha-\beta) \in \{1, \sqrt{2}, \sqrt{3}\}$.

  Our work builds on our recent resolution of the problem in the case of $\alpha = -\beta$ (corresponding to equiangular lines). It is the first determination of $\lim_{d \to \infty} N_{\alpha,\beta}(d)/d$ for any nontrivial fixed values of $\alpha$ and $\beta$ outside of the equiangular lines setting.
\end{abstract}

\maketitle

\section{Introduction}

A set of unit vectors in $\R^d$ is a \emph{spherical two-distance set} if the inner products of distinct vectors only take two values.
The problem of determining the maximum size of spherical two-distance sets is a deep and natural problem in discrete geometry. Some of the earliest results in this area date to the seminal work of Delsarte, Goethals, and Seidel \cite{DGS77}. They prove that a spherical two-distance set in $\R^d$ has size at most $\tfrac12 d(d+3)$.
This bound is close to the truth, as taking the $\frac12 d(d+1)$ midpoints on the edges of a regular simplex form a spherical two-distance set in $\R^d$.
Recently Glazyrin and Yu \cite{GY18} determined that the maximum size of spherical two-distance sets in $\R^d$ is indeed $\frac12 d(d+1)$ whenever $d \ge 7$ and $d+3$ is not an odd perfect square; see \cite{M09,BY13,Y17} for results in many small dimensions.

Given a set $A \subset [-1,1)$, a \emph{spherical $A$-code} is a set $S$ of unit vectors in $\R^d$ where $\langle x,y\rangle \in A$ for all distinct $x,y \in S$.
We write $N_A(d)$ the maximum size of a spherical $A$-code in $\R^d$.
In this paper, we are primarily interested in the case $A = \{\alpha, \beta\}$ for fixed $-1 \le \beta < \alpha < 1$ and large $d$, in which case we write $N_{\alpha,\beta}(d)$ instead of $N_{\{\alpha,\beta\}}(d)$.

Let us briefly mention some early developments on this problem.
The special case $\alpha = -\beta$ corresponds to \emph{equiangular lines}, whose study in the setting of fixed angle in high dimensions began with the work of Lemmens and Seidel~\cite{LS73}.
For spherical two-distance sets with fixed angles, Neumaier~\cite[Corollary~5]{N81} showed that $N_{\alpha,\beta}(d)\leq 2d+1$ unless $(1-\alpha)/(\alpha - \beta)$ is an integer. Furthermore, a result of Larman, Rogers, and Seidel~\cite{LRS77} implies the growth rate $N_{\alpha,\beta}(d) = \Theta_{\alpha,\beta}(d^2)$ for all $0\leq\beta<\alpha<1$ such that $(1-\alpha)/(\alpha - \beta)$ is an integer.\footnote{Using Wilson's deep result~\cite{W75} on the existence of balanced incomplete block designs, Larman, Rogers and Seidel~\cite[Theorem~3]{LRS77} constructed a spherical $\set{0,1/(\lambda+1)}$-code $C_{\lambda}(d)$ of size $\Theta_\lambda(d^2)$ in $\R^d$ for any positive integer $\lambda$, from which one constructs a spherical $\set{\alpha, \beta}$-code $\{\sqrt{1-\beta}(\bm v, \sqrt{\beta/(1-\beta)})\colon \bm v \in C_{\lambda}(d)\}$ of size $\Theta_\lambda(d^2)$ in $\R^{d+1}$ for every $0 \le \beta < \alpha$ with $\lambda = (1-\alpha)/(\alpha-\beta)$ a positive integer.} The regime $-1\le\beta<\alpha<0$ is less interesting, as an easy argument shows that $N_{[-1,\alpha]}(d) \le 1 - 1/\alpha$ for all $\alpha < 0$.

Recently work \cite{B16,BDKS18,JP20} culminated in a solution~\cite{JTYZZ21} to the problem of determining the maximum number of equiangular lines with fixed angles in high dimensions.
The papers~\cite{B16,BDKS18} also address the more general problem of estimating $N_{[-1,\beta]\cup\{\alpha_1, \dots, \alpha_k\}}(d)$ for fixed $\beta< 0 < \alpha_1 < \cdots < \alpha_k$.
In particular, Bukh \cite{B16} showed that $N_{[-1,\beta]\cup\{\alpha\}}(d)  = O_\beta(d)$, in sharp contrast to the quadratic dependence in dimension without angle restrictions.
Significant progress was made by Balla, Dr\"{a}xler, Keevash, and Sudakov \cite{BDKS18}, whose results in particular imply the bound $N_{\alpha,\beta}(d) \le 2(1-\alpha/\beta+o(1))d$.
More generally, it was conjectured in \cite{B16} and proved in \cite[Theorem 1.4]{BDKS18} that $N_{[-1,\beta]\cup\{\alpha_1, \dots, \alpha_k\}} = O_{k,\beta}(d^k)$, and that there exist choices of $\alpha_1, \dots, \alpha_k, \beta$ for which this upper bound is tight up to a constant factor. Here subscripts in the asymptotic notation indicate that  hidden constants may depend on these parameters.

We focus our attention on the goal of sharpening the above results for spherical two-distance sets and obtaining tight asymptotics for their maximum sizes.

\begin{problem} \label{prob:main}
  Determine, for fixed $-1\le\beta<0\le\alpha<1$, and large $d$, the maximum number, denoted $N_{\alpha,\beta}(d)$, of unit vectors in $\R^d$ whose pairwise inner products lie in $\{\alpha, \beta\}$. In particular determine the limit of $N_{\alpha,\beta}(d)/d$ as $d \to\infty$.
\end{problem}

We recently solved \cref{prob:main} in the case of equiangular lines \cite{JTYZZ21} where $\beta = -\alpha$. To state the result, we need the following spectral graph quantity, introduced in~\cite{JP20}.

\begin{definition} \label{def:k(lambda)}
  The \emph{spectral radius order}, denoted $k(\lambda)$, of a real number $\lambda > 0$ is  the smallest integer $k$ so that there exists a $k$-vertex graph $G$ whose spectral radius $\lambda_1(G)$ is exactly $\lambda$. Set $k(\lambda) = \infty$ if no such graph exists. (When we talk about the spectral radius or eigenvalues of a graph we always refer to its adjacency matrix.)
\end{definition}

\begin{theorem}[Equiangular lines with a fixed angle \cite{JTYZZ21}] \label{thm:equiangular}
  Fix $\alpha \in (0,1)$.
  Let $\lambda= (1-\alpha)/(2\alpha)$.
  For all sufficiently large $d > d_0(\alpha)$,
  \[
  N_{\alpha,-\alpha}(d) = \begin{cases}
    \displaystyle \floor*{\frac{k(\lambda) (d-1)}{k(\lambda)-1}} \qquad & \text{if } k(\lambda) < \infty, \\
    d + o(d) & \text{otherwise.}
  \end{cases}
  \]
\end{theorem}

Let us recap some key steps in the proof of the upper bound on $N_{\alpha,-\alpha}(d)$ in \cref{thm:equiangular}.
We will build on this framework.

Given a spherical $\set{\pm\alpha}$-code $S$, we consider the \emph{associated graph} $G$ with vertex set $S$, where $x,y \in S$ are adjacent in $G$ if $x \cdot y = -\alpha$.
We are allowed to replace any $x \in S$ by $-x$ without changing the equiangular lines configuration.
An argument introduced in \cite{BDKS18} reduces the problem to bounded degree graphs.

\begin{theorem}[{\cite{BDKS18} and \cite[Theorem 2.1]{JTYZZ21}}]
  For every $\alpha \in (0,1)$, there exists $\Delta$ depending only on $\alpha$, such that given any spherical $\set{\pm\alpha}$-code $S$ in $\R^n$, one can replace some subset of vectors in $S$ by their negations so that the associated graph $G$ (as defined above) has maximum degree at most $\Delta$.
\end{theorem}

The problem of bounding the size of $S$ is related to the multiplicity of $(1-\alpha)/(2\alpha)$ as the second largest eigenvalue of the adjacency matrix of $G$.
A crucial contribution of \cite{JTYZZ21} is that every connected bounded degree graph has sublinear second eigenvalue multiplicity. More generally, we have the following. (See \cref{def:signed-eig} below for the precise definition of $j$-th eigenvalue multiplicity.)

\begin{theorem}[{\cite[Theorem~2.2]{JTYZZ21}}] \label{thm:kth-ev-mult}
  For every $j$ and $\Delta$, there is a constant $C = C(\Delta,j)$ so that every connected $n$-vertex graph with maximum degree at most $\Delta$ has $j$-th eigenvalue multiplicity at most $Cn/\log\log n$.
\end{theorem}

Turning to spherical two-distance sets, 
given a spherical $\set{\alpha,\beta}$-code $S$ 
(with $\beta < 0 \le \alpha$ as always throughout this paper), 
we define its \emph{associated graph} $G$ to have vertex set $S$ and where $x,y \in S$ are adjacent in $G$ if $x \cdot y = \beta$. 
Unlike for equiangular lines, 
here we are no longer allowed to negate a subset of vectors in a spherical $\set{\alpha,\beta}$-code. 
Instead, we show that $G$ is very close to a complete $p$-partite graph.
Here $p$ is a specific constant, with the equiangular lines problem corresponding to $p=2$.

\begin{definition}
  A graph $G$ is a \emph{$\Delta$-modification} of another graph $H$ on the same vertex set if the symmetric difference of $G$ and $H$ has maximum degree at most $\Delta$.
\end{definition}

\begin{theorem} \label{thm:structure}
  For every $-1 \le \beta < 0 \le \alpha < 1$, there exists $\Delta$ depending only on $\alpha$ and $\beta$ such that for every spherical $\set{\alpha,\beta}$-code, its associated graph $G$ (as defined above), after removing at most $\Delta$ vertices, is a $\Delta$-modification of a complete $p$-partite graph, where $p = \floor{-\alpha/\beta}+1$.
\end{theorem}

\begin{remark}
  We allow empty parts in a complete $p$-partite graph. In particular, a complete $t$-partite graph is always a complete $p$-partite graph for $t \le p$.
\end{remark}

It will be helpful to study such graphs using the language of signed graphs.

\begin{definition} \label{def:signed-eig}
A \emph{signed graph} is a graph whose edges are each labeled by $+$ or $-$.
Throughout the paper we decorate variables for signed graphs with the $\pm$ superscript.
The \emph{signed adjacency matrix} $A_{G^\pm}$ of a signed graph $G^\pm$ on $n$ vertices is the $n\times n$ matrix whose $(i,j)$-th entry is $1$ if $ij$ is a positive edge, and $-1$ if $ij$ is a negative edge, and $0$ otherwise. We denote the eigenvalues of $A_{G^\pm}$ by $\lambda_1(G^\pm) \ge \lambda_2(G^\pm) \ge \cdots \ge \lambda_n(G^\pm)$. We write 
\[
\mult(\lambda,G^\pm) = \abs{\set{i: \lambda_i(G^\pm) = \lambda}}
\]
for the the multiplicity of $\lambda$ as an eigenvalue of $G^\pm$. 
The \emph{$j$-th eigenvalue multiplicity} of $G^\pm$ is defined to be $\mult(\lambda_j(G^\pm),G^\pm)$. 
We use $\abs{G}$ and $\abs{G^\pm}$ to denote the number of vertices in the graph.
\end{definition}

Given a $\Delta$-modification $G$ of a complete $p$-partite graph $K$, 
we study the signed graph $G^\pm$ defined by $A_{G^\pm} = A_G - A_{K}$.
The growth rate of $N_{\alpha,\beta}(d)$ is related to the eigenvalue multiplicity of $G^\pm$.
We introduce the following parameter generalizing the spectral radius order $k(\lambda)$ for signed graphs.

\begin{definition}
  A \emph{valid $p$-coloring} of a signed graph $G^\pm$ is a coloring of the vertices using $p$ colors such that the endpoints of every negative edge are colored using distinct colors, and the endpoints of every positive edge are colored using identical colors. (See \cref{fig:coloring} for an example.)
  The \emph{chromatic number} $\chi(G^\pm)$ of a signed graph $G^\pm$ is the smallest $p$ for which $G^\pm$ has a valid $p$-coloring. If $G^\pm$ does not have a valid $p$-coloring for any $p$, we write $\chi(G^\pm) = \infty$.
\end{definition}

\begin{definition}
  Given $\lambda > 0$ and $p \in \N$, define the parameter
  \[
    k_p(\lambda)=\inf\left\{\frac{\abs{G^\pm}}{\mult(\lambda, G^\pm)}\colon \chi(G^\pm) \le p \text{ and }\lambda_1(G^\pm) = \lambda\right\}.
  \]
  We say that $k_p(\lambda)$ is \emph{achievable} if it is finite and the infimum can be attained.
\end{definition}

\begin{figure}[t]
  \centering
  \begin{tikzpicture}[very thick,scale=.8]
    \coordinate (a) at (0.922649731,3);
    \coordinate (b) at (2.077350269,3);
    \coordinate (c) at (-0.577350269,1);
    \coordinate (d) at (0.577350269,1);
    \coordinate (e) at (0,0);
    \coordinate (f) at (2.422649731,1);
    \coordinate (g) at (3.577350269,1);
    \coordinate (h) at (3,0);
    \draw [fill=red!30,draw=none] (0,0.65) circle[radius=1];
    \draw [fill=blue!30,draw=none] (3,0.65) circle[radius=1];
    \draw [fill=green!30,draw=none] (1.5,3) circle[radius=1];
    \draw (a) -- (b);
    \draw (c) -- (e) -- (d);
    \draw (f) -- (g) -- (h) -- cycle;
    \draw[dashed] (c) -- (a) -- (d);
    \draw[dashed] (g) -- (b) -- (f);
    \draw[dashed] (d) -- (h);
    \draw[dashed] (e) -- (f);
    \draw[fill=black, draw] (1,3) circle[radius=.075];
    \draw[fill=black, draw] (2,3) circle[radius=.075];
    \draw[fill=black, draw] (-0.5774,1) circle[radius=.075];
    \draw[fill=black, draw] (0.5774,1) circle[radius=.075];
    \draw[fill=black, draw] (0,0) circle[radius=.075];
    \draw[fill=black, draw] (2.4226,1) circle[radius=.075];
    \draw[fill=black, draw] (3.5774,1) circle[radius=.075];
    \draw[fill=black, draw] (3,0) circle[radius=.075];
  \end{tikzpicture}
  \caption{A valid $3$-coloring of a signed graph. Throughout this paper, the positive edges are represented by solid segments and the negative edges are represented by dashed segments.}
  \label{fig:coloring}
\end{figure}
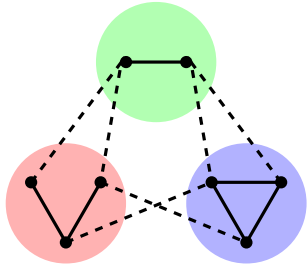

In the definition of $k_p(G^\pm)$, it is enough to consider connected $G^\pm$, since the eigenvalues of $G^\pm$ are given by the union of the the eigenvalues of its connected components.

If $\chi(G^\pm) \le 2$, then the signed graph $G^\pm$ and its underlying graph $G$ have the same eigenvalues (including multiplicities), since the signed adjacency matrix of $G^\pm$ can be obtained by conjugating the adjacency matrix of $G$ by a $\set{\pm 1}$-valued diagonal matrix.
By the Perron--Frobenius theorem, the top eigenvalue of a connected unsigned graph has multiplicity one. Thus,
\[
k_1(\lambda)=k_2(\lambda) = k(\lambda) \qquad \text{for all } \lambda > 0.
\]
However the behavior of $k_p(\lambda)$ is far more mysterious when $p \ge 3$. We do not know any general method of estimating or certifying values of $k_p(\lambda)$. Also, it is not even clear whether the infimum in the definition of $k_p(\lambda)$ can always be attained whenever $k_p(\lambda)$ is finite.

Generalizing the construction in \cite{JP20} relating equiangular lines to $k(\lambda)$, we can obtain a lower bound on $\lim_{d\to\infty}N_{\alpha,\beta}(d)/d$ (see \cref{prop:lower-bound}). Our main conjecture, below, says that this lower bound is sharp.

\begin{conjecture} \label{conj:main}
  Fix $-1\le\beta<0\le\alpha<1$.
  Set $\lambda = (1-\alpha)/(\alpha-\beta)$ and $p = \floor{-\alpha/\beta} + 1$. Then
  \[
    \lim_{d \to \infty}
    \frac{N_{\alpha,\beta}(d)}{d} = \begin{cases}
      \displaystyle \frac{k_p(\lambda) }{k_p(\lambda)-1} \qquad & \text{if }k_p(\lambda) < \infty, \\
      1 & \text{otherwise}.
    \end{cases}
  \]
\end{conjecture}

We see above that the parameters
\[
\lambda = \frac{1-\alpha}{\alpha-\beta}
\qquad\text{and}\qquad
p = \left\lfloor - \frac{\alpha}{\beta} \right\rfloor +1
\]
appear to play important roles in the problem.
These two parameters $\lambda$ and $p$ conjecturally govern the asymptotic behavior of $N_{\alpha,\beta}(d)$.
Our main theorem below establishes \cref{conj:main} for $p \le 2$, as well as for $\lambda \in \{1, \sqrt{2}, \sqrt{3}\}$.
This is the first time that some $\lim_{d \to \infty} N_{\alpha,\beta}(d)/d$ is determined outside of the equiangular lines setting ($\alpha = -\beta$).

\begin{theorem} \label{thm:main}
  Fix $-1\le\beta<0\le\alpha<1$.
  Set $\lambda = (1-\alpha)/(\alpha-\beta)$ and $p = \floor{-\alpha/\beta}+1$.
  \begin{enumerate}[label=(\alph*)]
    \item If $p \le 2$, then the maximum size $N_{\alpha,\beta}(d)$ of a spherical $\set{\alpha, \beta}$-code in $\R^d$ satisfies
    \[
    N_{\alpha,\beta}(d) = \begin{cases}
      \displaystyle \frac{k(\lambda) d}{k(\lambda)-1} + O_{\alpha,\beta}(1) \qquad & \text{if }k(\lambda) < \infty, \\
      d + o(d) & \text{otherwise}.
    \end{cases}
    \] \label{thm:main-p}
    \item If $\lambda = 1$ and $p \ge 2$, then $k_p(1) = p/(p-1)$ and $N_{\alpha,\beta}(d) = pd + O_{\alpha,\beta}(1)$. \label{thm:main-lambda-1}
    \item If $\lambda = \sqrt{3}$ and $p = 3$, then $k_3(\sqrt{3}) = 7/3$ and $N_{\alpha,\beta}(d) = 7d/4 + O_{\alpha,\beta}(1)$. \label{thm:main-triangle-counting}
    \item If $\lambda \in \set{\sqrt{2}, \sqrt{3}}$ and $p \ge \lambda^2 + 1$, then $k_p(\lambda) = 2$ and $N_{\alpha,\beta}(d) = 2d + O_{\alpha,\beta}(1)$. \label{thm:main-lambda-2-3}
  \end{enumerate}
  Moreover, $k_p(\lambda)$ is achievable for every $\lambda \in \{1,\sqrt{2}, \sqrt{3}\}$ and $p \in \N$.
\end{theorem}

\begin{remark}
  The conditions on $\lambda$ and $p$ in \cref{thm:main} can be directly translated to ones on $\alpha$ and $\beta$. The condition $p \le 2$ in \ref{thm:main-p} amounts to $\alpha + 2\beta < 0$, which includes the special case $\alpha = -\beta$ for equiangular lines. The conditions in both \ref{thm:main-lambda-1} and \ref{thm:main-lambda-2-3} amount to $(\lambda+1)\alpha - \lambda\beta = 1$ and $\lambda/(\lambda^2 +\lambda +1) \le \alpha < 1/(\lambda+1)$, where $\lambda \in \set{1,\sqrt{2},\sqrt{3}}$. For example, $(\alpha, \beta) = (2/5, -1/5)$ satisfies the last two conditions for $\lambda = 1$, yielding $N_{2/5,-1/5}(d) = 3d + O(1)$. It is worth contrasting the last example to the universal equiangular lines bound $N_{\alpha,-\alpha}(d) \le 2d + O_\alpha(1)$ for all fixed $\alpha > 0$ (implied by \cref{thm:equiangular}, but proved initially in \cite{BDKS18}). Lastly the condition in \ref{thm:main-triangle-counting} amounts to $(\sqrt{2}+1)\alpha-\sqrt 2\beta = 1$ and $2/(3\sqrt{2}+2)\le \alpha < 3/(4\sqrt{2}+3)$.
\end{remark}

We also prove a general upper bound on $N_{\alpha,\beta}(d)$, though it is not expected to be tight except for special values (e.g., it implies  \cref{thm:main}\ref{thm:main-p}\ref{thm:main-lambda-1}).

\begin{theorem} \label{thm:upper-bound}
  Fix $-1 \le \beta < 0 \le \alpha < 1$. Set $\lambda = (1-\alpha)/(\alpha-\beta)$ and $p = \floor{- \alpha / \beta} + 1$ and $q = \max\set{1, p/2}$. Then
  \[
    N_{\alpha,\beta}(d) \le \begin{cases}
      \displaystyle \frac{q k(\lambda) d}{k(\lambda) - 1} + O_{\alpha,\beta}(1) \qquad & \text{if }k(\lambda) < \infty, \\
      q d + o(d) & \text{otherwise}.
    \end{cases}
  \]
\end{theorem}

Our proof of \cref{thm:main} indeed confirms \cref{conj:main} in all the solved cases, namely when $p \le 2$ or $\lambda \in \{1,\sqrt{2},\sqrt{3}\}$.
We employ a number of different methods for bounding eigenvalue multiplicities in signed graphs in the different parts of \cref{thm:main}:

\begin{itemize}
  \item For \ref{thm:main-p} and \ref{thm:main-lambda-1}, we apply the sublinear bound on eigenvalue multiplicity of bounded degree unsigned graphs (\cref{thm:kth-ev-mult} above; see \cref{sec:graph-mult}).
  \item For \ref{thm:main-triangle-counting}, we develop a forbidden induced subgraph framework (see \cref{sec:forbidden}), and we apply a careful third moment and triangle counting argument (see \cref{sec:triangle-counting}).
  \item For \ref{thm:main-lambda-2-3} we apply an algebraic degree argument (see \cref{sec:algebraic}). Additionally, we confirm \cref{conj:main} for all algebraic integers $\lambda$ whose degree equals $k(\lambda)$ (see the end of \cref{sec:algebraic}).
\end{itemize}

\begin{remark}
  After this work is completed, building on our forbidden induced subgraph framework, Jiang and Polyanskii~\cite{JP21} proved \cref{conj:main} for every $\lambda < \lambda^*$, where $\lambda^* = \beta^{1/2} + \beta^{-1/2} \approx 2.01980$ and $\beta$ is the unique real root of $x^3 = x + 1$.
\end{remark}

A major obstacle to completely settling \cref{conj:main} is that bounded degree signed graphs may have linear top eigenvalue multiplicity.

\begin{theorem} \label{thm:large-mult-exmaple}
  For every $n \ge 3$, there is a connected signed graph with $6n$ vertices, maximum degree $5$, and chromatic number $3$, such that its largest eigenvalue appears with multiplicity $n$.
\end{theorem}

The rest of the paper is organized as follows. In \cref{sec:spectral-graph}, we explain the connection with spherical two-distance sets and the spectral theory of signed graphs, and further proves a lower bound on $N_{\alpha,\beta}(d)$. 
In \cref{sec:struct-proof} we prove the structural result, \cref{thm:structure}.
In \cref{sec:graph-mult} we prove \cref{thm:main}\ref{thm:main-p}, \cref{thm:main}\ref{thm:main-lambda-1}, and \cref{thm:upper-bound} using \cref{thm:kth-ev-mult}. In \cref{sec:forbidden} we develop a forbidden induced subgraph framework to bound $N_{\alpha,\beta}(d)$ from above. In \cref{sec:triangle-counting} we prove \cref{thm:main}\ref{thm:main-triangle-counting} via a third moment argument under the forbidden induced subgraph framework. In \cref{sec:algebraic} we prove \cref{thm:main}\ref{thm:main-lambda-2-3} via an algebraic argument. In \cref{sec:large-mult} we give two constructions related to \cref{thm:large-mult-exmaple}.

\section{Connection to spectral theory of signed graphs} \label{sec:spectral-graph}

The spherical two-distance set problem has the following equivalent spectral graph theoretic formulation. Here $A \succeq 0$ means that $A$ is positive semidefinite.

\begin{lemma} \label{lem:code-graph}
  Let $-1\le\beta<\alpha<1$.
  Set $\lambda = (1-\alpha)/(\alpha-\beta)$ and $\mu = \alpha/(\alpha-\beta)$. There exists a spherical $\set{\alpha, \beta}$-code of size $N$ in $\R^d$ if and only if there exists a graph $G$ on $N$ vertices satisfying
  \[
    \lambda I-A_{G}+\mu J \succeq 0 \qquad\text{and} \qquad \rnk(\lambda I-A_{G}+\mu J) \le d.
  \]
\end{lemma}

\begin{proof}
  For a spherical $\set{\alpha, \beta}$-code $\set{v_1, \dots, v_N}$ in $\R^d$, let $G$ be the \emph{associated graph} on vertex set $\set{1, \dots, N}$, where $ij$ is an edge whenever $\langle v_i, v_j\rangle = \beta$. The Gram matrix $M = (\langle v_i, v_j\rangle)_{i,j}$ has 1's on its diagonal and $\alpha, \beta$ everywhere else, so it equals $(1-\alpha)I - (\alpha - \beta)A_G + \alpha J$, where $I$ is the identity matrix, $J$ the all-ones matrix, and $A_G$ the adjacency matrix of $G$. We have $M/(\alpha - \beta) = \lambda I - A_G + \mu J$, where $\lambda = (1-\alpha)/(\alpha - \beta)$ and $\mu = \alpha/(\alpha - \beta)$. Since the Gram matrix $M$ is positive semidefinite and has rank at most $d$, the same holds for $\lambda I - A_G + \mu J$.

  Conversely, for every $G$, $\lambda$ and $\mu$ for which $\lambda I-A_{G}+\mu J$ is positive semidefinite and has rank $d$, there exists a corresponding configuration of $N$ unit vectors in $\R^d$, with pairwise inner products in $\set{\alpha,\beta}$.
\end{proof}

We are now ready to establish a lower bound on $N_{\alpha,\beta}(d)$ using \cref{lem:code-graph}.

\begin{proposition} \label{prop:lower-bound}
  Fix $-1\le\beta<0\le\alpha<1$. Then $N_{\alpha,\beta}(d) \ge d$ for every positive integer $d$. Moreover if $k_p(\lambda) < \infty$, where $\lambda = (1-\alpha)/(\alpha-\beta)$ and $p = \floor{-\alpha / \beta} + 1$, then
  \[
    N_{\alpha,\beta}(d) \ge \begin{cases}
      \displaystyle \frac{k_p(\lambda) d}{k_p(\lambda)-1} - O_{\alpha,\beta}(1) \qquad & \text{if }k_p(\lambda)\text{ is achievable}, \\
      \displaystyle \frac{k_p(\lambda) d}{k_p(\lambda)-1} - o(d) & \text{otherwise}.
    \end{cases}
  \]
\end{proposition}

\begin{proof}
  Let  $\mu = \alpha/(\alpha-\beta)$.
  Take $G$ to be $d$-vertex graph with no edges, so that $A_G = 0$ and $\lambda I - A_G + \mu J$ is positive semidefinite and has rank at most $d$. So $N_{\alpha,\beta}(d) \ge d$ by \cref{lem:code-graph}. In fact, the spherical two-distance set constructed here forms a regular $(d-1)$-simplex.

  Hereafter assume that $k_p(\lambda) < \infty$. We first construct, for every signed graph $G^\pm$ with $\chi(G^\pm) \le p$ and $\lambda_1(G^\pm) = \lambda$, a spherical $\set{\alpha, \beta}$-code of size $\abs{G^\pm}$ in dimension $\abs{G^\pm} - \mult(\lambda, G^\pm) + p$. Let $V_1, \dots, V_p$ be the color classes of a valid $p$-coloring.
  Consider the unsigned graph $G$ obtained from taking the symmetric difference between the underlying graph of $G^\pm$ and the complete $p$-partite graph with parts $V_1, \dots, V_p$. The adjacency matrix of $G$ is related to the signed adjacency matrix of $G^\pm$ by
  \[
    A_{G} = A_{G^\pm} + A_K,
  \]
  where $K$ is the complete $p$-partite graph with parts $V_1, \dots, V_p$. Therefore,
  \[
    \lambda I - A_{G} + \mu J = (\lambda I - A_{G^\pm}) + (\mu J - A_K).
  \]
  We have $\lambda I - A_{G^\pm} \succeq 0$ since $\lambda_1(G^\pm) = \lambda$. We now note that $\mu J - A_K$ is positive semidefinite. Indeed, for every $\bm{x}\in \R^{V(G^\pm)}$, we set $s_i = \sum_{v\in V_i} \bm{x}_v$ for each $i \in \set{1,\dots, p}$, and we see
  \[
    \bm{x}^\T(\mu J - A_K)\bm{x} = \mu \left(\sum_i s_i\right)^2 - \sum_{i\neq j}s_is_j = \mu\sum_i s_i^2 - (1 - \mu)\sum_{i\neq j}s_is_j.
  \]
  Because $\sum_{i\neq j}s_is_j \leq (p - 1)\sum_i s_i^2$ and $p \le {1}/({1 - \mu})$, we conclude that
  \[
    \bm{x}^\T(\mu J - A_K)\bm{x}\geq \left(\mu - (1-\mu)\left(\frac{1}{1-\mu}-1\right)\right)\sum_i s_i^2 = 0.
  \]
  Therefore $\mu J - A_K \succeq 0$, and so $\lambda I - A_{G} + \mu J \succeq 0$. We conclude by \cref{lem:code-graph} that there exists a spherical $\set{\alpha, \beta}$-code of size $\abs{G^\pm}$ in $\mathbb{R}^{d}$, where
  \[
    d = \rnk(\lambda I - A_{G} + \mu J) \leq \rnk{ (\lambda I - A_{G^\pm})} + \rnk(\mu J - A_K) \leq \abs{G^\pm} - \mult(\lambda, G^\pm) + p.
  \]

  Now fix an arbitrary $\eps > 0$. Take a signed graph $G^\pm_\eps$ such that ${\abs{G^\pm_\eps}}/{\mult(\lambda, G^\pm_\eps)} \le k_p(\lambda) + \eps$, $\chi(G^\pm_\eps) \le p$, and $\lambda_1(G^\pm_\eps) = \lambda$. For each positive integer $\ell$, denote by $\ell G^\pm_\eps$ the disjoint union of $\ell $ copies of $G^\pm_\eps$. We have $\abs{\ell G^\pm_\eps} = \ell\abs{G^\pm_\eps}$, $\mult(\lambda, \ell G^\pm) = \ell \mult(\lambda, G^\pm)$, $\chi(\ell G^\pm_\eps) = \chi(G^\pm_\eps) \le p$ and $\lambda_1(\ell G^\pm_\eps) = \lambda_1(G^\pm_\eps) = \lambda$. Thus we can apply the above construction to $G^\pm = \ell G^\pm_\eps$ to obtain a spherical $\set{\alpha, \beta}$-code of size $\ell \abs{G^\pm_\eps}$ in dimension $\ell (\abs{G^\pm_\eps} - \mult(\lambda, G^\pm_\eps)) + p$. We conclude that
  \begin{multline*}
    N_{\alpha,\beta}(d) \geq \abs{G^\pm_\eps}\floor*{\frac{d - p}{\abs{G^\pm_\eps} - \mult(\lambda, G^\pm_\eps)}} \ge \frac{d}{1 - \mult(\lambda, G^\pm_\eps) / \abs{G^\pm_\eps}} - O_{\alpha,\beta,\eps}(1) \\
    \ge \frac{d}{1-1/(k_p(\lambda)+\eps)} - O_{\alpha,\beta,\eps}(1) = \frac{(k_p(\lambda) + \eps)d}{k_p(\lambda)-1 + \eps} - O_{\alpha,\beta,\eps}(1).
  \end{multline*}
  
  Finally notice that when $k_p(\lambda)$ is achievable, we can take $\eps = 0$ in the above argument.
\end{proof}

\section{Structure of the associated graph} \label{sec:struct-proof}

In this section we prove \cref{thm:structure}, which gives a structure characterization of graphs that can arise from a spherical two-distance set. To that end, we introduce the following notation.

\begin{definition}
  Given a graph $G$, for sets $Y \subseteq X \subseteq V(G)$, define $C_X(Y)$ to be the set of vertices in $V(G) \setminus X$ that are adjacent to all vertices in $Y$ and not adjacent to any vertices in $X\setminus Y$, and for a set $X \subseteq V(G)$ and $\Delta \in \N$, define
  \[
    C_{X,\Delta} = \bigcup_{Y\subseteq X\colon \abs{Y} \le \Delta} C_{X}(Y)
    \qquad \text{and} \qquad
    C_{X,-\Delta} = \bigcup_{Y\subseteq X\colon \abs{X \setminus Y} \le \Delta} C_{X}(Y).
  \]
\end{definition}

We now present a series of structural lemmas leading to the proof of \cref{thm:structure}.

\begin{lemma} \label{lem:struct}
  For every $\lambda>0$ and $\mu\in(0,1)$, there exist $\Delta \in \N$ and $L_0 \in \N$ such that for every graph $G$ that satisfies $\lambda I - A_G + \mu J \succeq 0$ the following holds.
  \begin{enumerate}[label=(\alph*)]
    \item Neither of the following is an induced subgraph of $G$:
    \begin{enumerate}[label=(a\arabic*)]
      \item the complete graph $K_\Delta$; \label{lem:struct-a1}
      \item the complete $(p+1)$-partite graph $K_{\Delta, \dots, \Delta}$, where $p=\floor{1/(1-\mu)}$. \label{lem:struct-a2}
    \end{enumerate}
    \item For every independent set $X$ of size $L$ in $G$, if $L \ge L_0$, then
    \begin{enumerate}[label=(b\arabic*)]
      \item the maximum degree of $G[C_{X,\Delta}]$ is less than $\Delta$, and \label{lem:struct-b1}
      \item the number of vertices not in $C_{X,\Delta}\cup C_{X,-\Delta}$ is at most $L2^L$. \label{lem:struct-b2}
    \end{enumerate}
    \item For every pair of disjoint vertex subsets $X_1$ and $X_2$, each of size $L$, in $G$, if $L \ge L_0$ and $G[X_1 \cup X_2]$ is the complete bipartite graph with parts $X_1$ and $X_2$, then
    \begin{enumerate}[label=(c\arabic*)]
      \item every vertex in $C_{X_1, \Delta} \cap C_{X_2, -\Delta}$ is adjacent to all but at most $\Delta$ vertices in $C_{X_1, -\Delta} \cap C_{X_2, \Delta}$, and \label{lem:struct-c1}
      \item the number of vertices in $C_{X_1,\Delta}\cap C_{X_2,\Delta}$ is less than $\Delta$. \label{lem:struct-c2}
    \end{enumerate}
  \end{enumerate}
\end{lemma}

\begin{proof}[Proof of \ref{lem:struct-a1}]
  Suppose on the contrary that $G$ contains $K_\Delta$ as a subgraph. Let $\bm v\in \R^{V(G)}$ be the vector that assigns $1$ to vertices in $K_\Delta$ and $0$ otherwise. Then $\bm{v}^\T(\lambda I - A_G + \mu J) \bm{v}$ becomes
  \[
    \lambda \Delta - \Delta(\Delta - 1) + \mu \Delta^2,
  \]
  which would be negative if we had chosen $\Delta > (1+\lambda)/(1-\mu)$.
\end{proof}
\begin{proof}[Proof of \ref{lem:struct-a2}]
  Suppose on the contrary that $G$ contains the complete $(p+1)$-partite graph $K_{\Delta, \dots, \Delta}$ as an induced subgraph. Again let $\bm v\in \R^{V(G)}$ be the vector that assigns $1$ to the vertices in $K_{\Delta, \dots, \Delta}$ and $0$ otherwise. Then $\bm{v}^\T(\lambda I - A_G + \mu J) \bm{v}$ becomes
  \[
    \lambda (p+1)\Delta - p(p+1)\Delta^2 + \mu((p+1)\Delta)^2 = (p+1)\Delta\left(\lambda - (p - \mu(p+1))\Delta\right).
  \]
  Because $p > 1/(1-\mu) - 1 = \mu/(1-\mu)$ or equivalently $p > \mu(p+1)$, the last factor above would be negative if we had chosen $\Delta > \lambda/(p-\mu(p+1))$.
\end{proof}
\begin{proof}[Proof of \ref{lem:struct-b1}]
  Suppose on the contrary that a vertex $u \in C_{X, \Delta}$ has $\Delta$ neighbors $v_1, \dots, v_\Delta \in C_{X, \Delta}$. Let $\bm{v}\in \R^{V(G)}$ be the vector that assigns $L$ to $u$, $\lambda L/\Delta$ to $v_1, \dots, v_\Delta$, $-(\lambda+1)$ to the vertices in $X$, and $0$ otherwise. Because $u, v_1, \dots, v_\Delta \in C_{X, \Delta}$, we have
  \[
    \tfrac{1}{2}\bm{v}^\T A_G\bm{v} \ge \lambda L^2 - (\lambda+1)\Delta L - \lambda(\lambda+1)\Delta L = \lambda L^2 - (\lambda+1)^2\Delta L.
  \]
  Using this bound and the fact that $\bm{v}^\T \bm 1 = 0$, we obtain that $\bm{v}^\T(\lambda I - A_G + \mu J) \bm{v}$ is at most
  \[
    \lambda (L^2 + \lambda^2L^2/\Delta + (\lambda+1)^2L) - 2(\lambda L^2 - (\lambda+1)^2\Delta L)= -\lambda(1-\lambda^2/\Delta) L^2 + O_{\lambda, \Delta}(L),
  \]
  which would be negative for sufficiently large $L$ if we had chosen $\Delta > \lambda^2$.
\end{proof}
\begin{proof}[Proof of \ref{lem:struct-b2}]
  To show that $\abs{V(G)\setminus (C_{X,\Delta}\cup C_{X,-\Delta})} \le L2^L$, it suffices to prove $\abs{C_X(A)} < L$ for every subset $A$ of the independent set $X$ such that $\abs{A} > \Delta$ and $\abs{X \setminus A} > \Delta$.

  Write $a = \abs{A}$, $b = \abs{X \setminus A}$, and $c = \abs{C_X(A)}$. For any $\alpha, \beta, \gamma \in \R$, we consider the vector $\bm v\in \R^{V(G)}$ that assigns $\alpha$ to the vertices in $A$, $\beta$ to the vertices in $X\setminus A$,  $\gamma$ to the vertices in $C_X(A)$, and $0$ otherwise, and we have
  \[
    0 \le \bm v^\T(\lambda I - A_G + \mu J)\bm v \le \lambda(a\alpha^2 +b\beta^2 + c\gamma^2) - 2ac\alpha\gamma + \mu(a\alpha + b\beta + c\gamma)^2.
  \]
  In particular, taking $\beta = -(a\alpha + c\gamma)/(b + \lambda/\mu)$, we obtain that for all $\alpha, \gamma \in \R$,
  \[
  0 \le \lambda(a\alpha^2  + c\gamma^2) - 2ac\alpha\gamma + \frac{\mu\lambda}{\mu b + \lambda}(a\alpha + c\gamma)^2.
  \]
  For this quadratic form in $\alpha$ and $\gamma$ to be positive semidefinite, its discriminant must be nonpositive:
  \[
    \frac{( \mu b+\lambda(1 - \mu))^2}{(\mu b+\lambda)^2}a^2c^2 - \left(\lambda a + \frac{\mu\lambda a^2}{\mu b+\lambda}\right)\left(\lambda c + \frac{\mu\lambda c^2}{\mu b+\lambda}\right)\leq 0,
  \]
  which simplifies to
  \begin{equation} \label{eqn:abc}
    (\mu b+\lambda(1 - \mu))^2 ac \leq \lambda^2(\mu a + \mu b + \lambda)(\mu b + \mu c + \lambda).
  \end{equation}
  By the assumption that $a, b > \Delta$, if we had taken $\Delta \geq \max\set{\lambda / \mu, 4\lambda^2, 2}$, then $\lambda < \mu b$ and $\lambda^2 < b/4$, hence \cref{eqn:abc} would imply the following series of inequalities:
  \begin{equation*}
    \mu^2 ab^2c < (b/4)(\mu a + 2\mu b)(2\mu b + \mu c) \implies abc < (a+b)(b+c) \implies c < \frac{(a + b)b}{ab - a - b} \leq a + b = L.\qedhere
  \end{equation*}
\end{proof}
\begin{proof}[Proof of \ref{lem:struct-c1}]
  Suppose on the contrary that a vertex $u \in C_{X_1, \Delta} \cap C_{X_2, -\Delta}$ is not adjacent to $v_1, \dots, v_\Delta \in C_{X_1, -\Delta} \cap C_{X_2, \Delta}$. Let $\bm{v}\in \R^{V(G)}$ be the vector that assigns $L$ to $v$, $-\lambda L/\Delta$ to $v_1,\ldots,v_\Delta$, $-1$ to the vertices in $X_1$, $\lambda$ to the vertices in $X_2$, and $0$ otherwise. Because $u \in C_{X_1, \Delta} \cap C_{X_2, -\Delta}$ and $v_1, \dots, v_\Delta \in C_{X_1, -\Delta} \cap C_{X_2, \Delta}$, we have
  \[
    \tfrac{1}{2}\bm v^\T A_G \bm v \ge -\Delta L + \lambda(L-\Delta)L + (\lambda L / \Delta)\Delta(L-\Delta) - \lambda L^2 = \lambda L^2 - (2\lambda + 1)\Delta L.
  \]
  Using this bound and the fact that $\bm v^\T \bm 1 = 0$, we obtain that $\bm v^\T (\lambda I - A_G + \mu J) \bm v$ is at most
  \[
    \lambda(L^2 + \lambda^2L^2/\Delta + L + \lambda^2L) - 2(\lambda L^2 - (2\lambda + 1)\Delta L) = -\lambda(1-\lambda^2/\Delta)L^2 + O_{\lambda,\Delta}(L),
  \]
  which would be negative for sufficiently large $L$ if we had chosen $\Delta > \lambda^2$.
\end{proof}
\begin{proof}[Proof of \ref{lem:struct-c2}]
  Suppose on the contrary that $C_{X_1, \Delta} \cap C_{X_2, \Delta}$ contains $v_1, \dots, v_\Delta$.
  Let $\bm v\in \R^{V(G)}$ be the vector that assigns $2L/\Delta$ to $v_1, \dots, v_\Delta$, $-1$ to the vertices in $X_1 \cup X_2$, and $0$ otherwise. Because $v_1, \dots, v_\Delta \in C_{X_1, \Delta} \cap C_{X_2, \Delta}$, we have
  \[
    \tfrac{1}{2}\bm v^\T A_G v \ge -(2L/\Delta)2\Delta^2 + L^2 = -4\Delta L+L^2.
  \]
  Using this bound and the fact that $\bm v^\T \bm 1 = 0$, we obtain that $\bm v^\T (\lambda I - A_G + \mu J) \bm v$ is at most
  \[
    \lambda(4L^2 / \Delta + 2L) - 2(-4\Delta L + L^2) = -2(1-2\lambda / \Delta)L^2 + O_{\lambda,\Delta}(L),
  \]
  which would be negative for sufficiently large $L$ if we had chosen $\Delta > 2\lambda$.
\end{proof}

\begin{proof}[Proof of \cref{thm:structure}]
  Let $\lambda = (1-\alpha)/(\alpha-\beta)$ and $\mu = \alpha/(\alpha-\beta)$ (and so $p = \floor{-\alpha/\beta+1} = \floor{1/(1-\mu)}$).
  As in \cref{lem:code-graph}, 
  the associate graph $G$ of the spherical $\set{\alpha,\beta}$-set satisfies $\lambda I-A_{G}+\mu J \succeq 0$.

  Choose $\Delta$ and $L_0$ as in \cref{lem:struct}. We shall prove that $G$, after removing at most $pL2^L +\binom{p}{2}\Delta + R(\Delta, L2^{pL})$ vertices, is a $p\Delta$-modification of a complete $p$-partite graph, where $L = L_0 + (p + 2)\Delta$ and $R(\cdot, \cdot)$ is the Ramsey number.
  
  We may assume that $\abs{G} \ge R(\Delta, L)$ because otherwise $G$ is vacuously a $p\Delta$-modification of a complete $p$-partite graph after removing all its vertices. By \cref{lem:struct}\ref{lem:struct-a1} and Ramsey's theorem, there exists an independent set of size $L$ in $G$.
  Choose the maximum $t \le p$ such that the complete $t$-partite graph $K_{L -t\Delta, \dots, L -t\Delta}$ is an induced subgraph of $G$ 
  (note that $t\ge 1$ since there is an independent set of size $L$).
  Let $X_1, \dots, X_t \subset V(G)$ be the parts of this $t$-partite graph.
  
  Define for every $i \in \set{1, \dots, t}$ the vertex subset
  \[
    V_i = C_{X_i,\Delta}\cap\bigcap_{j\neq i}C_{X_j,-\Delta}.
  \]
  By \ref{lem:struct-b1} and \ref{lem:struct-c1} in \cref{lem:struct}, we see that the $G[V_1 \cup \dots \cup V_t]$ is a $t\Delta$-modification of the complete $t$-partite graph with parts $V_1, \dots, V_t$.

  We bound $U := V(G) \setminus (V_1 \cup \dots \cup V_t)$ as follows. Set
  \[
    U_i = V(G) \setminus (C_{X_i,\Delta}\cup C_{X_i,-\Delta}), \quad U_{ij}^- = C_{X_i, \Delta} \cap C_{X_j, \Delta}, \quad U^+ = \bigcap_{i} C_{X_i,-\Delta}.
  \]
  Note that $U = (\bigcup_i U_i) \cup (\bigcup_{i < j} U_{ij}^-) \cup U^+$. It is enough to bound the cardinalities of $U_i, U_{ij}, U^+$.
  \cref{lem:struct}\ref{lem:struct-b2} says that $\abs{U_i} \le L2^L$ for each $i$. \cref{lem:struct}\ref{lem:struct-c2} says that $\abs{U_{ij}^-} \le \Delta$ for $i < j$. 
  
  Finally, we claim that $U^+$ does not contain a subset of size $L 2^{tL}$ that is independent in $G$. Indeed, suppose on the contrary that $U^+$ contains an independent set of size $L 2^{tL}$. Since every vertex in $U^+$ has at least $L - (t+1)\Delta$ neighbors in $X_i$ for each $i$, by the pigeonhole principle, there exist $X_1' \subseteq X_1, \dots, X_t' \subseteq X_t$ and $U' \subseteq U^+$, each of size $L - (t + 1)\Delta$, such that $G[X_1' \cup \dots \cup X_t' \cup U']$ is a complete $(t+1)$-partite graph with parts $X_1', \dots, X_t'$ and $U'$, which contradicts our choice of $t$ or \cref{lem:struct}\ref{lem:struct-a2} in case $t = p$. This finishes the proof of the claim. In view of \cref{lem:struct}\ref{lem:struct-a1} and Ramsey's theorem, we obtain $\abs{U^+} < R(\Delta, L2^{tL})$. In total, $\abs{U} \le tL2^L + \binom{t}{2}\Delta + R(\Delta, L2^{tL})$.
\end{proof}

\section{Graph eigenvalue multiplicity argument} \label{sec:graph-mult}

We estimate the eigenvalue multiplicity of a signed graph with bounded maximum degree by that of a (not necessarily connected) graph. Recall \cref{def:k(lambda)} of the spectral radius order $k(\lambda)$.

\begin{lemma} \label{lem:mult-bound}
  For every $\lambda > 0$, $\Delta \in \N$, and $j \in \N$, if $G$ is an $n$-vertex graph with  maximum degree at most $\Delta$ and $\lambda_j(G) \le \lambda$, then
  \[
  \mult(\lambda, G) \le \begin{cases}
    n / k(\lambda) + O_{\Delta,j,\lambda}(1) \qquad & \text{if }k(\lambda) < \infty,\\
    O_{\Delta,j}(n / \log\log n) & \text{otherwise}.
  \end{cases}
  \]
\end{lemma}
\begin{proof}
  Let $G_1, \dots, G_t$ be the connected components of $G$ numbered such that $\lambda_1(G_1), \dots, \lambda_1(G_s) > \lambda$ and  $\lambda_1(G_{s+1}), \dots, \lambda_1(G_t) \le \lambda$. Because $\lambda_j(G) \le \lambda$, we know that $s < j$. Set $n_i = \abs{G_i}$ and $n = \sum n_i = \abs{G}$.

  For each $i \le s$, since $G_i$ is a connected graph with maximum degree at most $\Delta$ and $\lambda_j(G_i) \le \lambda$, \cref{thm:kth-ev-mult} gives a constant $C = C(\Delta,j)$ such that
  \begin{equation} \label{eqn:mult-1}
    \mult(\lambda, G_i) \le \frac{Cn_i}{\log\log n_i}.
  \end{equation}
  
  We break the rest of the proof into two cases.

  \medskip
  \noindent\textbf{Case $k(\lambda) < \infty$.} Set $N_0 = \exp(\exp (Ck(\lambda)))$. For $i \le s$, when $n_i \ge N_0$, we can relax \cref{eqn:mult-1} to $\mult(\lambda,G_i) \le n_i / k(\lambda)$; when $n_i < N_0$, clearly $\mult(\lambda,G_i) \le n_i < N_0$. To sum up, for $i \le s$, we always have
  \begin{equation} \label{eqn:mult-2}
    \mult(\lambda, G_i) \le \frac{n_i}{k(\lambda)} + N_0.
  \end{equation}

  For each $i > s$, when $\lambda_1(G_i) = \lambda$, because $G_i$ is connected, we know that $n_i \ge k(\lambda)$, and so by the Perron--Frobenius theorem, we obtain
  \begin{equation} \label{eqn:mult-3}
    \mult(\lambda, G_i) \le 1 \le \frac{n_i}{k(\lambda)};
  \end{equation}
  when $\lambda_1(G_i) < \lambda$, clearly \cref{eqn:mult-3} holds trivially.
  We combine \cref{eqn:mult-2} and \cref{eqn:mult-3} to obtain
  \[
    \mult(\lambda,G) = \sum_{i=1}^t\mult(\lambda, G_i) \le \sum_{i=1}^t \frac{n_i}{k(\lambda)} + s N_0 \le \frac{n}{k(\lambda)} + O_{\Delta,j,\lambda}(1).
  \]

  \medskip
  \noindent\textbf{Case $k(\lambda) = \infty$.} For $i > s$, because $\lambda_1(G_i) \le \lambda$ and $k(\lambda) = \infty$, it must be the case that $\lambda_1(G_i) < \lambda$, and so $\mult(\lambda,G_i) = 0$. Therefore \cref{eqn:mult-1} gives
  \begin{equation*}
    \mult(\lambda,G) = \sum_{i=1}^s\mult(\lambda,G_i) \le j \cdot \max_{1 \le i \le j} \frac{Cn_i}{\log\log n_i} = O_{\Delta,j} \left(\frac{n}{\log\log n}\right). \qedhere
  \end{equation*}
\end{proof}

Next we prove \cref{thm:upper-bound}, which states that
\[
  N_{\alpha,\beta}(d) \le \begin{cases}
    \displaystyle \frac{q k(\lambda) d}{k(\lambda) - 1} + O_{\alpha,\beta}(1) \qquad & \text{if }k(\lambda) < \infty, \\
    q d + o(d) & \text{otherwise},
  \end{cases}
\]
where
$\lambda = (1-\alpha)/(\alpha-\beta)$ and $p = \floor{- \alpha / \beta} + 1$ and $q = \max\set{1, p/2}$.

\begin{proof}[Proof of \cref{thm:upper-bound}]
  In view of \cref{lem:code-graph}, consider a graph $\wt{G}$ on $N_{\alpha,\beta}(d)$ vertices satisfying
  \[
    \lambda I - A_{\wt{G}} + \mu J \succeq 0 \quad\text{and}\quad {\rank}{\left(\lambda I - A_{\wt{G}} + \mu J\right)} \le d,
  \]
  where $\lambda = (1-\alpha)/(\alpha - \beta)$ and $\mu = \alpha / (\alpha - \beta)$. By \cref{thm:structure} we obtain a constant $\Delta = \Delta(\alpha,\beta)$ such that the graph, denoted $G$, obtained from $\wt{G}$ by removing at most $\Delta$ vertices is a $\Delta$-modification of a complete $p$-partite graph, denoted $K$, where $p = \floor{1/(1-\mu)}$. Define the signed graph $G^\pm$ by $A_{G^\pm} = A_G - A_K$. Notice that the maximum degree of $G^\pm$ is at most $\Delta$, and $\chi(G^\pm) \le p$.
  
  Now the signed adjacency matrix of $G^\pm$ satisfies
  \[
    \lambda I - A_{G^\pm} + \mu J - A_K \succeq 0 \quad\text{and}\quad {\rank}{\left(\lambda I - A_{G^\pm} + \mu J - A_K\right)} \le d.
  \]
  Note that $\rank(\mu J - A_K) \le p$. From the first condition above, we deduce using the Courant--Fischer theorem that $\lambda_{p+1}(\lambda I - A_{G^\pm}) \ge 0$ or equivalently $\lambda_{p+1}(G^\pm) \le \lambda$. From the second condition above, we deduce using subadditivity of matrix ranks that $\rank(\lambda I - A_{G^\pm}) \le d + p$ or equivalently
  \begin{equation} \label{eqn:mlt-gpm}
    \mult(\lambda, G^\pm) \ge \abs{G^\pm} - (d+p).
  \end{equation}

  We break the rest of the proof into two cases.

  \medskip
  \noindent\textbf{Case $p = 1$.} The signed graph $G^\pm$ consists of positive edges only. \cref{lem:mult-bound} provides the upper bound
  \[
    \mult(\lambda, G^\pm) \le \begin{cases}
      \abs{G^\pm} / k(\lambda) + O_{\alpha,\beta}(1) & \text{if }k(\lambda) < \infty, \\
      o(\abs{G^\pm}) & \text{otherwise}.
    \end{cases}
  \]
  Combining with \cref{eqn:mlt-gpm}, we get
  \[
    \abs{G^\pm} - (d + p) \le \begin{cases}
      \abs{G^\pm}/k(\lambda) + O_{\alpha, \beta}(1) & \text{if }k(\lambda) < \infty, \\
      o(\abs{G^\pm}) & \text{otherwise},
    \end{cases}
  \]
  which implies
  \[
    \abs{G^\pm} \le \begin{cases}
      \displaystyle \frac{k(\lambda)d}{k(\lambda)-1} + O_{\alpha,\beta}(1) & \text{if }k(\lambda) < \infty, \\
      d + o(d) & \text{otherwise}.
    \end{cases}
  \]
  The desired upper bound on $N_{\alpha,\beta}(d)$ follows immediately in view of $\abs{G^\pm} \ge N_{\alpha,\beta}(d) - \Delta$.
  
  \medskip
  \noindent\textbf{Case $p \ge 2$.} Let $V_1$ and $V_2$ be the largest parts of the complete $p$-partite graph $K$. Let $G^\pm_{12}$ be the signed subgraph of $G^\pm$ induced on $V_1 \cup V_2$, and let $G_{12}$ be the underlying graph of $G_{12}^\pm$. Notice that $\abs{G_{12}} = \abs{V_1} + \abs{V_2} \ge 2\abs{G^\pm}/p$, and the maximum degree of $G_{12}$ is at most $\Delta$, and $\chi(G^\pm_{12}) \le 2$.
  Since $\chi(G^\pm_{12}) \le 2$, the signed graph $G_{12}^\pm$ is isospectral to its underlying graph $G_{12}$. It follows from \cref{lem:mult-bound} that
  \[
    \mult(\lambda, G_{12}^\pm) = \mult(\lambda, G_{12}) \le \begin{cases}
      \abs{G_{12}}/k(\lambda) + O_{\alpha,\beta}(1) & \text{if }k(\lambda) < \infty, \\
      o(\abs{G_{12}}) & \text{otherwise}.
    \end{cases}
  \]
  By the Cauchy interlacing theorem, we have
  \[
    \mult(\lambda, G^\pm) -  (\abs{G^\pm} - \abs{G_{12}}) \le \mult(\lambda, G_{12}^\pm).
  \]
  Combining \cref{eqn:mlt-gpm} and the above two inequalities, we get
  \[
    \abs{G_{12}} - (d + p) \stackrel{\cref{eqn:mlt-gpm}}{\le} \mult(\lambda, G^\pm) -  (\abs{G^\pm} - \abs{G_{12}}) \le \begin{cases}
      \abs{G_{12}}/k(\lambda) + O_{\alpha,\beta}(1) & \text{if }k(\lambda) < \infty, \\
      o(\abs{G_{12}}) & \text{otherwise},
    \end{cases}
  \]
  which implies
  \begin{equation*}
    \abs{G_{12}} \le \begin{cases}
      \displaystyle \frac{k(\lambda)d}{k(\lambda)-1} + O_{\alpha,\beta}(1) & \text{if }k(\lambda) < \infty, \\
      d + o(d) & \text{otherwise}.
    \end{cases}
  \end{equation*}
  The desired upper bound on $N_{\alpha,\beta}(d)$ follows immediately in view of the inequalities $\abs{G_{12}} \ge 2\abs{G^\pm}/p$ and $\abs{G^\pm} \ge N_{\alpha,\beta}(d) - \Delta$.
\end{proof}

As a corollary, we obtain the following general lower bound on $k_p(\lambda)$.

\begin{corollary} \label{prop:kp-lower-bound}
  For all $\lambda > 0$ and $p \ge 2$,
  \[
    k_p(\lambda) \ge \frac{pk(\lambda)}{pk(\lambda)-2\lambda}.
  \]
\end{corollary}

\begin{proof}
  Comparing \cref{prop:lower-bound} and \cref{thm:upper-bound}, we get
  \begin{equation*}
    \frac{k_p(\lambda)d}{k_p(\lambda)-1} - o(d) \le \frac{pk(\lambda)d}{2(k(\lambda)-1)} + O_{p,\lambda}(1),
  \end{equation*}
  which implies the desired lower bound. (It is also not hard to prove \cref{prop:kp-lower-bound} directly, but we do not do so here.)
\end{proof}

\begin{figure}
  \centering
  \begin{minipage}[t]{0.5\textwidth}
    \centering
    \begin{tikzpicture}[scale=2, very thick]
      \draw (0,0) circle[radius=1];
      \foreach \r in {-30,90,210} {
        \draw[rotate=\r] (-40:1) -- (40:1);
        \draw[rotate=\r] (200:1) -- (0:1);
        \draw[rotate=\r] (0:1) -- (160:1);
        \draw[fill, rotate=\r] (40:1) circle[radius=0.03];
        \draw[fill, rotate=\r] (0:1) circle[radius=0.03];
        \draw[fill, rotate=\r] (-40:1) circle[radius=0.03];
      }
    \end{tikzpicture}
    \captionof{figure}{The Paley graph of order $9$.} \label{fig:paley-9}
  \end{minipage}%
  \begin{minipage}[t]{0.5\textwidth}
    \centering
    \begin{tikzpicture}[scale=2, very thick]
      \draw (0,0) circle[radius=1];
      \foreach \r in {0,45,90,135,180,225,270,315} {
        \draw[rotate=\r] (112.5:1) -- (22.5:1);
        \draw[rotate=\r] (22.5:1) -- (67.5:0.5);
        \draw[rotate=\r] (67.5:0.5) -- (112.5:1);
        \draw[rotate=\r] (112.5:0.5) -- (22.5:0.5);
        \draw[rotate=\r] (22.5:0.5) -- (157.5:0.5);
        \draw[fill, rotate=\r] (22.5:1) circle[radius=0.03];
        \draw[fill, rotate=\r] (22.5:0.5) circle[radius=0.03];
      }
    \end{tikzpicture}
    \captionof{figure}{The Shrikhande graph.} \label{fig:shrikhande}
  \end{minipage}
\end{figure}

\begin{remark}
  For general $\lambda$, we do not know any algorithm for computing $k(\lambda)$ (or even deciding whether $k(\lambda) <\infty$), though deciding whether $k(\lambda) < k$ for each integer $k$ is a finite problem as can be done by a brute-force search over all graphs up to a fixed size.

  When $\lambda \in \N$, we have $k(\lambda) = \lambda + 1$ because the complete graph $K_{\lambda + 1}$ is the graph on fewest vertices with spectral radius $\lambda$. In contrast, even for $\lambda \in \N$, computing the exact values of $k_p(\lambda)$ seems to be very difficult for $p \ge 3$. For $\lambda = 2$, \cref{prop:kp-lower-bound} implies that $k_3(2) \ge 9/5$ and $k_4(2) \ge 3/2$. Note that both the Paley graph of order $9$ in \cref{fig:paley-9} and the Shrikhande graph in \cref{fig:shrikhande} are strongly regular graphs with $-2$ as their smallest eigenvalue with multiplicity $4$ and $9$ respectively. Moreover their chromatic numbers are $3$ and $4$ respectively. The all-negative signed graphs of these two strongly regular graphs would yield $k_3(2) \le 9/4$ and $k_4(2) \le 16/9$. We leave the determination of $k_p(2)$ for $p \ge 3$ as an open problem.
\end{remark}

\cref{thm:main}\ref{thm:main-p} and \cref{thm:main}\ref{thm:main-lambda-1} follow easily from \cref{thm:upper-bound,prop:lower-bound}.

\begin{proof}[Proof of \cref{thm:main}\ref{thm:main-p}]
  Because $p \le 2$, we have $q = \max\set{1, p/2} = 1$ and $k_p(\lambda) = k(\lambda)$. Moreover, if $k(\lambda) < \infty$ then $k(\lambda)$ can be achieved for $k_p(\lambda)$ by the smallest graph whose spectral radius is exactly $\lambda$. Thus \cref{thm:upper-bound,prop:lower-bound} give matching bounds on $N_{\alpha, \beta}(d)$.
\end{proof}

\begin{proof}[Proof of \cref{thm:main}\ref{thm:main-lambda-1}]
  Because $\lambda = 1$ and $p \ge 2$, we have $k(\lambda) = 2$ and $q = \max(1, p/2) = p/2$. Thus \cref{thm:upper-bound} gives
  \begin{equation} \label{eqn:pd-plus-o1}
    N_{\alpha,\beta}(d) \le pd + O_{\alpha,\beta}(1).
  \end{equation}

  \cref{prop:kp-lower-bound} implies that $k_p(1) \ge p/(p-1)$. To see that $p/(p-1)$ can be achieved for $k_p(1)$, consider the all-negative complete signed graph $K_p^\pm$ on $p$ vertices. Clearly $\chi(K_p^\pm) = p$. Since the smallest eigenvalue of the complete unsigned graph $K_p$ is $-1$ with multiplicity $p-1$, the largest eigenvalue of $K_p^\pm$ is $1$ with multiplicity $p-1$. Now \cref{prop:lower-bound} provides a lower bound that matches \cref{eqn:pd-plus-o1} up to an additive constant.
\end{proof}

\section{Forbidden induced subgraphs} \label{sec:forbidden}

The next lemma enables us to forbid finitely many induced subgraphs in the signed graph that arises from \cref{thm:structure}. Here an \emph{induced subgraph} of a signed graph keeps the original edge signs.

\begin{lemma} \label{lem:additional-structure}
  Fix $\lambda > 0$, $\mu \in (0,1)$, $p \in \N$, and $\Delta \in \N$. For every signed graph $H^\pm$ with $\lambda_1(H^\pm) > \lambda$, there exists $n_0 \in \N$ such that for every $t \le p$ and every graph $G$ that is a $\Delta$-modification of a complete $t$-partite graph $K$, if $\lambda I - A_G + \mu J \succeq 0$, and the size of each part of $K$ is at least $n_0$, then $H^\pm$ cannot be an induced subgraph of the signed graph $G^\pm$ defined by $A_{G^\pm} = A_G - A_K$.
\end{lemma}

\begin{proof}
  Suppose that $G$ is a $\Delta$-modification of a complete $t$-partite graph $K$ with parts $\wt{V}_1, \dots, \wt{V}_t$, and suppose that the size of each part of $K$ is at least $n_0$. Assume for the sake of contradiction that $H^\pm$ with $\lambda_1(H^\pm) > \lambda$ is an induced subgraph of $G^\pm$. Take $n_0 = (\abs{H^\pm} + pm)\Delta$, where $m = \floor{\lambda\abs{H^\pm}/(\lambda_1(H^\pm) - \lambda)} + 1$. We can greedily find $V_1 \subseteq \wt{V}_1, \dots, V_t\subseteq \wt{V}_t$ such that
  \begin{enumerate}
    \item each $V_i$ is disjoint from $V(H^\pm)$ and has size $m$,
    \item $G$ induces a complete $t$-partite graph with parts $V_1, \dots, V_t$,
    \item for every vertex $v$ of $H^\pm$, if $v \in \wt{V}_i$, then, in $G$, the vertex $v$ is adjacent to every vertex in $V_j$ for $j \neq i$, and is not adjacent to any vertex in $V_i$.
  \end{enumerate}

  Let $\bm{x} \in \R^{V(H^\pm)}$ be a top eigenvector of $H^\pm$, and set
  \[
    s_i = \sum_{u \in V(H^\pm) \cap \wt{V}_i}\bm{x}_u.
  \]
  Note that $s_i^2 \le \abs{V(H^\pm) \cap \wt{V}_i} \bm{x}^\intercal \bm{x}$ for each $i$, which implies that
  \begin{equation} \label{eqn:sum-si}
    \sum_i s_i^2 \le \abs{H^\pm} \bm{x}^\intercal \bm{x}.
  \end{equation}
  Consider the vector $\bm{v} \in \R^{V(G)}$ extending $\bm{x}$ that in addition assigns $-s_i/m$ to each vertex in $V_i$ for $i \in \set{1, \dots, t}$. Since $\bm{v}$ is chosen so that $\sum_{u \in \wt{V}_i}\bm v_u = 0$ for each $i\in\set{1,\dots,t}$, we have $J\bm{v}=0$ and $A_K\bm{v}=0$.
  Now we can simplify the quadratic form as follows:
  \[\bm{v}^\T(\lambda I - A_G + \mu J)\bm{v} = \bm{v}^\T(\lambda I - A_{G^\pm} - A_K + \mu J)\bm{v} = \bm{v}^\T(\lambda I - A_{G^\pm})\bm{v}.\]
  Next, since no vertex in $H^\pm$ is adjacent to $V_1\cup\cdots\cup V_t$ in $G^\pm$, we have
  \begin{align*}
     \bm{v}^\T(\lambda I - A_{G^\pm})\bm{v} & =\bm{x}^\T(\lambda I - A_{H^\pm})\bm{x} + \lambda\sum_i m(s_i/m)^2 \\
    & = \left(\lambda - \lambda_1(H^\pm)\right)\bm{x}^\T\bm{x} + \lambda\sum_i s_i^2 / m \\ & \le \left(\lambda - \lambda_1(H^\pm) + \lambda\abs{H^\pm}/m\right)\bm{x}^\T\bm{x},\tag*{by \cref{eqn:sum-si}}
  \end{align*}
  which is negative because $m > \lambda\abs{H^\pm}/(\lambda_1(H^\pm) - \lambda)$. This contradicts $\lambda I-A_G+\mu J \succeq 0$.
\end{proof}

\cref{lem:additional-structure} leads us to bound eigenvalue multiplicities in a restricted class of signed graphs obtained by forbidding certain induced subgraphs.

\begin{definition} \label{def:M}
  Given a family $\mathcal H$ of signed graphs, let $M_{p, \mathcal H}(\lambda, N)$ be the maximum possible value of $\mult(\lambda, G^\pm)$ over all signed graphs $G^\pm$ on at most $N$ vertices that do not contain any member of $\mathcal H$ as an induced subgraph and satisfy $\chi(G^\pm) \le p$ and $\lambda_{p+1}(G^\pm) \le \lambda$.
\end{definition}

In our application, we will only be allowed to forbid a \emph{finite} $\mathcal H$ such that $\lambda_1(H^\pm) > \lambda$ for all $H^\pm \in \mathcal H$.

\begin{remark}
  We could choose $\mathcal H$ properly so that every signed graph $G^\pm$ considered in \cref{def:M} of $M_{p, \mathcal H}(\lambda, N)$ has its maximum degree bounded by a constant depending only on $p$ and $\lambda$. In fact, set $D = \floor{\lambda^2}$, and suppose that $\mathcal H$ includes all the signed graphs $H^\pm$ on $D+2$ vertices with $\chi(H^\pm) \le 2$ such that the underlying graph of $H^\pm$ contains the star $K_{1,D+1}$. One can then show that for every graph $G^\pm$ that does not contain any member of $\mathcal H$ as an induced subgraph, the maximum degree of $G^\pm$ is at most $\chi(G^\pm)D$.
\end{remark}

The next statement relates the maximum size of a spherical two-distance set with the above eigenvalue multiplicity quantity.

\begin{theorem} \label{thm:signed-mult-upper}
  Fix $-1 \le \beta < 0 \le \alpha < 1$.
  Set $\lambda = (1-\alpha)/(\alpha-\beta)$ and $p = \floor{-\alpha/\beta}+1$.
  Let $\mathcal H$ be a finite family of signed graphs with $\lambda_1(H^\pm) >\lambda$ for each $H^\pm \in \mathcal H$.
  Then
  \[
    N_{\alpha,\beta}(d) \le d + M_{p,\mathcal H}(\lambda, N_{\alpha,\beta}(d)) + O_{\alpha,\beta, \mathcal H}(1).
  \]
\end{theorem}

\begin{proof}
  In view of \cref{lem:code-graph}, consider a graph $\wt{G}$ on $N_{\alpha,\beta}(d)$ vertices satisfying
  \[
    \lambda I - A_{\wt{G}} + \mu J \succeq 0 \quad\text{and}\quad \rank(\lambda I - A_{\wt{G}} + \mu J ) \le d,
  \]
  where $\lambda = (1-\alpha)/(\alpha - \beta)$ and $\mu = \alpha/(\alpha-\beta)$. By \cref{lem:struct} we obtain a constant $\Delta = \Delta(\alpha, \beta)$ such that $\wt{G}$, after removing at most $\Delta$ vertices, is a $\Delta$-modification of a complete $p$-partite graph, where $p = \floor{1/(1-\mu)}$.

  Let $n_0 = n_0(\alpha, \beta, \mathcal{H})$ be the maximum $n_0$ given by \cref{lem:additional-structure} when it is applied to each member of $\mathcal H$ respectively with the parameters $\lambda$, $\mu$, $p$, and $\Delta$. After removing at most $\Delta$ vertices from $\wt{G}$, we can further remove at most $pn_0$ vertices from $\wt{G}$ to obtain a graph, denoted $G$, that is a $\Delta$-modification of a $t$-partite graph, denoted $K$, with each part of size at least $n_0$, for some $t \le p$. Define the signed graph $G^\pm$ by $A_{G^\pm} = A_G - A_K$. Since $\lambda I - A_G + \mu J \succeq 0$, by our choice of $n_0$, we know that the signed graph $G^\pm$ does not contain any member of $\mathcal H$ as an induced subgraph. Notice that $\chi(G^\pm) \le t \le p$.

  Now the signed adjacency matrix of $G^\pm$ satisfies
  \begin{subequations}
    \begin{gather}
      \lambda I - A_{G^\pm} + \mu J - A_K \succeq 0, \label{eqn:gpm-cond1} \\
      \rank(\lambda I - A_{G^\pm} + \mu J - A_K) \le d. \label{eqn:gpm-cond2}
    \end{gather}
  \end{subequations}
  Note that $\rank(\mu J - A_K) \le t \le p$. From \cref{eqn:gpm-cond1} we deduce using the Courant--Fischer theorem that $\lambda_{p+1}(\lambda I - A_{G^\pm}) \ge 0$ or equivalently $\lambda_{p+1}(G^\pm) \le \lambda$. Recall that $G^\pm$ has at most $N_{\alpha,\beta}(d)$ vertices, $G^\pm$ does not contain any member of $\mathcal H$ as an induced subgraph, and $\chi(G^\pm) \le p$. According to \cref{def:M},
  \[
    \mult(\lambda, G^\pm) \le M_{p, \mathcal H}(\lambda, N_{\alpha,\beta}(d)).
  \]
  
  From \cref{eqn:gpm-cond2} we deduce using subadditivity of matrix ranks that $\rank(\lambda I - A_{G^\pm}) \le d + p$ or equivalently
  \[
    \mult(\lambda, G^\pm) \ge \abs{G^\pm} - (d + p).
  \]
  Combining with $\abs{G^\pm} \ge N_{\alpha,\beta}(d) - \Delta - pn_0$, we get
  \begin{align*}
    N_{\alpha,\beta}(d) & \le \abs{G^\pm} + \Delta + pn_0 \\
    & \le d + \mult(\lambda,G^\pm) + \Delta + p(n_0+1) \\
    & \le d + M_{p, \mathcal H}(\lambda, N_{\alpha,\beta}(d)) + O_{\alpha,\beta,\mathcal H}(1). \qedhere
  \end{align*}
\end{proof}

For each value of $\lambda$ and $p$, if we could prove the following upper bound on the eigenvalue multiplicity, then it would imply \cref{conj:main} via \cref{thm:signed-mult-upper}.

\begin{conjecture} \label{conj:forbidden}
  For every $\lambda > 0$ and $p \in \N$, there exists a finite family $\mathcal{H}$ of signed graphs with $\lambda_1(H^\pm) >\lambda$ for each $H^\pm \in \mathcal H$ such that
  \[
    M_{p, \mathcal H}(\lambda, N) \le \begin{cases}
      N/k_p(\lambda)+o(N) \qquad & \text{if }k_p(\lambda) < \infty, \\
      o(N) & \text{otherwise}.
    \end{cases}
  \]
\end{conjecture}

We include the short deduction below that for each $\lambda > 0$ and $p \in \N$, \cref{conj:forbidden} implies \cref{conj:main}.
Though, for deducing \cref{thm:main}\ref{thm:main-triangle-counting} in the next section, we will prove each bound directly without resorting to \cref{conj:forbidden}, in order to give a slightly better error term of $O_{\alpha,\beta}(1)$ instead of $o(d)$.

\begin{proof}[Proof that \cref{conj:forbidden} implies \cref{conj:main} for each $\lambda > 0$ and $p\in \N$]
  Choose $\mathcal H$ as in \cref{conj:forbidden}.
  In the case when $k_p(\lambda) < \infty$, by \cref{thm:signed-mult-upper}, we have
  \[
  N_{\alpha,\beta}(d) \le d + M_{p, \mathcal H}(\lambda, N_{\alpha,\beta}(d)) + O_{\alpha, \beta}(1)
  \le d + \left(\frac{1}{k_p(\lambda)} + o(1)\right) N_{\alpha,\beta}(d).
  \]
  Therefore
  \[
  N_{\alpha,\beta}(d) \le \left( \frac{k_p(\lambda)}{k_p(\lambda) - 1} + o(1)\right) d,
  \]
  which matches the lower bound in \cref{prop:lower-bound}. The case of $k_p(\lambda) = \infty$ is similar.
\end{proof}

\section{Third moment argument} \label{sec:triangle-counting}

For $\lambda = \sqrt{3}$ and $p = 3$, we give a tight upper bound (verifying \cref{conj:forbidden}) on $\mult(\lambda, G^\pm)$ for those signed graphs $G^\pm$ in \cref{thm:signed-mult-upper}, which implies a tight upper bound on the corresponding $N_{\alpha,\beta}(d)$.

\begin{theorem} \label{prop:k3(sqrt3)}
  There exists a finite family $\mathcal H$ of signed graphs with $\lambda_1(H^\pm)> \sqrt{3}$ for each $H^\pm \in \mathcal H$ such that
  \[
    M_{3, \mathcal H}(\sqrt{3}, N) \le 3N/7.
  \]
\end{theorem}

\begin{proof}
  Let $\mathcal H$ be the family of all the signed graphs $H^\pm$ on at most $5$ vertices with $\lambda_1(H^\pm) > \sqrt{3}$. For the sake of contradiction, assume that $G^\pm$ is a signed graph with the minimum number of vertices such that $\chi(G^\pm) \le 3$, no member of $\mathcal H$ is an induced subgraph of $G^\pm$, and $\mult(\sqrt{3}, G^\pm) > 3\abs{G^\pm}/7$. By our choice of $\mathcal H$, every subgraph of $G^\pm$ induced by at most $5$ vertices has largest eigenvalue at most $\sqrt{3}$. Note that $G^\pm$ is connected by its minimality. Let $V(G^\pm) = V_1 \sqcup V_2 \sqcup V_3$ be a valid $3$-coloring of $G^\pm$ allowing some $V_i$'s being empty, and let $G$ be the underlying graph of $G^\pm$. The next four claims reveal the local structure of $G^\pm$.

  \begin{claim} \label{claim:triangle}
    The edges of every triangle in $G$ are all negative in $G^\pm$.
  \end{claim}

  \begin{claimproof}[Proof of \cref{claim:triangle}]
    Since $\chi(G^\pm)$ is finite, every signed triangle in $G^\pm$, other than the all negative one, contains $0$ or $2$ negative edges. In either case, the chromatic number of the signed triangle is $2$, hence its largest eigenvalue equals $\lambda_1(K_3) = 2$. However, every induced triangle of $G^\pm$ has largest eigenvalue at most $\sqrt{3}$.
  \end{claimproof}

  \begin{claim} \label{claim:local-3-star}
    If $G$ induces a star on $\set{v_0, v_1, v_2, v_3}$ centered at $v_0$, then $v_1, v_2, v_3$ are the only neighbors of $v_0$ in $G$, and moreover for every $w \neq v_0$ that is adjacent to at least one of $v_1, v_2, v_3$, exactly two of $v_1, v_2 ,v_3$ are adjacent to $w$ in $G$.
  \end{claim}

  \begin{claimproof}[Proof of \cref{claim:local-3-star}] Let $w \in V(G) \setminus \set{v_0, v_1, v_2, v_3}$ be a vertex that is adjacent to at least one of $v_0, v_1, v_2, v_3$, and consider the vector $\bm{v} \in \R^{W}$, where $W = \set{v_0,v_1,v_2,v_3,w}$, that assigns $\sqrt{3}$ to $v_0$, $\sigma(v_0v_i)$ to $v_i$ for $i \in \set{1,2,3}$, $\eps$ to $w$, where $\sigma\colon E(G) \to \set{\pm 1}$ is the signing of $G^\pm$ and $\eps \in \R$. According to our choice of $\bm v$, we have
  \[
    \bm{v}^\T A_{G^\pm[W]}\bm{v} = 6\sqrt3 + 2\eps\sum_{v_iw\in E(G)}\sigma(v_iw)\bm v_{v_i}.
  \]
  By the Courant--Fischer theorem, we also have
  \[
    \bm{v}^\T A_{G^\pm[W]}\bm{v} \le \lambda_1(G^\pm[W])\bm v^\T\bm v \le \sqrt{3}(6 + \eps^2).
  \]
  For the last inequality to hold for all $\eps \in \R$, we must have
  \[
    \sum_{v_iw\in E(G)}\sigma(v_iw)\bm v_{v_i} = 0,
  \]
  which implies that $v_0w \not\in E(G)$, and exactly two of $v_1, v_2 ,v_3$ are adjacent to $w$ in $G$.
  \end{claimproof}

  \begin{claim} \label{claim:max-degree-4}
    The maximum degree of $G$ is at most $4$.
  \end{claim}

  \begin{claimproof}[Proof of \cref{claim:max-degree-4}]
  Suppose on the contrary that $v_0$ is adjacent to at least $5$ vertices in $G$. Without loss of generality we may assume that $v_0 \in V_1$, and by the pigeonhole principle that $3$ neighbors, say $v_1, v_2, v_3$, of $v_0$ are in $V_1 \cup V_2$. As $\chi(H^\pm) \le 2$, where $H^\pm := G^\pm[\set{v_0, v_1, v_2, v_3}]$, by \cref{claim:triangle}, $H^\pm$ contains no triangles. Thus $G$ induces a star on $\set{v_0,v_1,v_2,v_3}$ centered at $v_0$, and so by \cref{claim:local-3-star}, $v_0$ has no neighbors other than $v_1, v_2, v_3$ in $G$, which leads to a contradiction.
  \end{claimproof}

  \begin{claim} \label{claim:induced-3-star}
    The underlying graph $G$ contains an induced star $K_{1,3}$.
  \end{claim}

  \begin{claimproof}[Proof of \cref{claim:induced-3-star}]
  Suppose on the contrary that $G$ does not contain any induced $K_{1,3}$. For every $v \in V(G)$, the subgraph of $G$ induced by the neighbors of $v$ contains no independent set of size $3$, in particular, this induced subgraph contains at most $2$ connected components, hence it contains at least $d_v - 2$ edges, where $d_v$ is the degree of $v$ in $G$. In other words, every $v \in V(G)$ is contained in at least $d_v - 2$ triangles.

  Recall from \cref{claim:local-3-star} that every triangle in $G$ has all its edges negatively signed. Let $\lambda_1, \lambda_2, \dots, \lambda_n$ be the eigenvalues of $G^\pm$, where $n = \abs{G^\pm}$, and let $t$ be the total number of triangles in $G$. Thus we have
  \[
    -\sum_i \lambda_i^3 = -\tr(A_{G^\pm}^3) = 6t \ge 2\sum_v (d_v - 2).
  \]
  Note that
  \[
    \sum_i \lambda_i^2 = \tr(A_{G^\pm}^2) = \sum_v d_v
    \qquad\text{and}\qquad
    \sum_i \lambda_i = \tr(A_{G^\pm}) = 0.
  \]
  Thus we have
  \[
    \sum_i (\lambda_i^3 + 2\lambda_i^2 - 7 \lambda_i) \le -2\sum_v (d_v - 2) + 2\sum_v d_v = 4n.
  \]
  Since the characteristic polynomial of $A_{G^\pm}$ is a polynomial with integer coefficients, we obtain $\mult(-\sqrt{3}, G^\pm) = \mult(\sqrt{3}, G^\pm)$, which is more than $3n/7$. For other eigenvalues $\lambda_i$, by \cref{claim:max-degree-4}, we know that $\lambda_i \ge -4$, and so
  \[
    \lambda_i^3 + 2\lambda_i^2 - 7 \lambda_i = (\lambda_i - 1)^2(\lambda_i + 4) - 4 \ge -4.
  \]
  Therefore
  \[
    \sum_i (\lambda_i^3 + 2\lambda_i^2 - 7 \lambda_i) > \frac{3n}{7}\cdot 2 \cdot 2(\sqrt3)^2 + \frac{n}{7}\cdot (-4) = \frac{32n}{7} > 4n,
  \]
  which is a contradiction.
  \end{claimproof}
  
  The following claim imposes restriction on $G^\pm$ with small number of vertices.

  \begin{claim} \label{claim:v-6-8}
    The number $n$ of vertices in $G$ is either $6$ or at least $8$. Moreover, if $n \in \set{6,8}$ then $G$ is a $3$-regular graph, and the signed adjacency matrix of $G^\pm$ satisfies $A_{G^\pm}^2 = 3I$.
  \end{claim}
  
  \begin{claimproof}
    \cref{claim:induced-3-star} shows that $n \ge 4$, and moreover when $n = 4$, $G$ is precisely $K_{1,3}$, in which case $\mult(\sqrt3, G^\pm) = \mult(\sqrt3, G) = 1 \le 3n/7$. Thus $n \ge 5$. Because $\mult(-\sqrt3, G^\pm) = \mult(\sqrt3, G^\pm) > 3n/7$, we obtain
    \begin{equation} \label{eqn:mult-sqrt3-ineq}
      n \ge \mult(\sqrt3, G^\pm) + \mult(-\sqrt3, G^\pm) \ge 2(\floor{3n/7}+1),
    \end{equation}
    which rules out $n=5$ and $n=7$. Therefore $n = 6$ or $n \ge 8$. Suppose that $n \in \set{6,8}$. Note that equality must hold for \cref{eqn:mult-sqrt3-ineq}. Thus $\mult(-\sqrt3, G^\pm) = \mult(\sqrt3, G^\pm) = n/2$, which implies that $\mult(3, A_{G^\pm}^2) = n$. Hence $A_{G^\pm}^2 = 3I$, which in particular implies that $G$ is a $3$-regular graph.
  \end{claimproof}

  Suppose $G$ induces a star on $\set{v_0, v_1, v_2, v_3}$ centered at $v_0$. Let $L_i$ be the set of vertices at distance $i$ from $v_0$ in $G$. From \cref{claim:local-3-star}, we know that $L_1 = \set{v_1, v_2, v_3}$, and moreover every $w \in L_2$ is adjacent to exactly two among $v_1, v_2, v_3$. Because $\abs{G} \ge 6$, it must be the case that $L_2 \neq \varnothing$. We break the rest of the proof into two cases.

  \medskip
  \noindent\textbf{Case $\abs{L_2} = 1$.} Suppose $L_2 = \set{w}$. By \cref{claim:local-3-star}, without loss of generality, $w$ is adjacent to $v_1$ and $v_2$. Because $\abs{G} \ge 6$, it must be the case that $L_3 \neq \varnothing$. Take any $w' \in L_3$. Note that $G$ induces a star on $\set{w, v_1, v_2, w'}$ centered at $w$. By \cref{claim:local-3-star}, $L_3 = \set{w'}$ and $L_4 = \varnothing$, which implies $\abs{G} = 6$. By \cref{claim:v-6-8}, $G$ is a $3$-regular graph, which is a contradiction.

  \medskip
  \noindent\textbf{Case $\abs{L_2} \ge 2$.} For every two $w_1, w_2 \in L_2$, we claim that they do not have the same pairs of neighbors in $L_1$.
  Indeed, suppose on the contrary that both $w_1$ and $w_2$ are, without loss of generality, adjacent to $v_1$ and $v_2$ in $L_1$. Since $v_2$ is adjacent to $v_0, w_1, w_2$, by \cref{claim:local-3-star}, $G$ does not induce a star on $\set{v_0, v_1, w_1, w_2}$ centered at $v_1$, and so $w_1w_2 \in E(G)$. Now we have two triangles $w_1w_2v_1$ and $w_1w_2v_2$, which by \cref{claim:triangle} all have negative edges. Thus $v_1$ and $v_2$ are in the same part of the valid $3$-coloring. 
  Let $H^\pm:= G^\pm[v_0,v_1,v_2,w_1]$. Then $\chi(H^\pm) \le 2$ and $H^\pm$ is a signed $4$-cycle. Thus $\lambda_1(H^\pm) = \lambda_1(C_4) = 2$, where $C_4$ denotes the $4$-cycle, contradicting to $\lambda_1(H^\pm) \le \sqrt3$.
  
  Assume for a moment that $\abs{G} = 6$. In this subcase, $\abs{L_2} = 2$ and $L_3 = \varnothing$, and so the degree of every vertex in $L_2$ is $2$. By \cref{claim:v-6-8}, $G$ is a $3$-regular graph, which is a contradiction. Hereafter $\abs{G} \ge 8$.
  
  Because no two vertices in $L_2$ have the same pairs of neighbors in $L_1$, $\abs{L_2} \le \binom{3}{2} = 3$. Because $\abs{G} \ge 8$, it must be the case that $L_3 \neq \varnothing$. Take $w_1 \in L_2$ and $w' \in L_3$ such that $w_1w' \in E(G)$. Without loss of generality, suppose that $w_1$ is adjacent to $v_1$ and $v_2$. Since $G$ induces a star on $\set{v_1, v_2, w_1, w'}$ centered at $w_1$, by \cref{claim:local-3-star}, $w'$ is the only neighbor of $w_1$ in $L_3$, and $w'$ has no neighbor in $L_4$. Now take an arbitrary vertex $w_2 \in L_2 \setminus \set{w_1}$. Since $w_1$ and $w_2$ do not have the same pairs of neighbors in $L_1$, the vertex $w_2$ is adjacent to only one of $v_1$ and $v_2$, and so $w_2w' \in E(G)$ by \cref{claim:local-3-star}. We can apply the previous argument to $w_2$ in place of $w_1$, and conclude that $w'$ is the only neighbor of $w_2$ in $L_3$. Since $w_2 \in L_1 \setminus \set{w_1}$ was chosen arbitrarily, we know that $L_3 = \set{w'}$ and $L_4 = \varnothing$, which implies $\abs{L_2} = 3$ and $\abs{G} = 8$.
  
  Since $G$ is a $3$-regular graph by \cref{claim:v-6-8}, it is easy to see that $G$ must be the cubical graph. In view of \cref{claim:v-6-8}, $G^\pm$ is a signed cube that satisfies $A_{G^\pm}^2 = 3I$, which means that every square of $G^\pm$ contains odd number of negative edges. Because $\chi(G^\pm) \le 3$, $G^\pm$ has no cycle with exactly one negative edge, and in particular every square of $G^\pm$ contains exactly one positive edge. At this point, it is not hard to deduce that $G^\pm$ is exactly $H_3^\pm$ in \cref{fig:signed-h3}. However $\chi(H_3^\pm) = 4$, which is a contradiction.
\end{proof}

\begin{proof}[Proof of \cref{thm:main}\ref{thm:main-triangle-counting}]
  \cref{thm:signed-mult-upper,prop:k3(sqrt3)} give
  \[
    N_{\alpha, \beta}(d) \le d + 3N_{\alpha, \beta}(d)/7 + O_{\alpha,\beta}(1),
  \]
  which implies 
  \begin{equation} \label{eqn:pd-plus-o2}
    N_{\alpha, \beta}(d) \le 7d/4 + O_{\alpha,\beta}(1).
  \end{equation}
  Comparing with \cref{prop:lower-bound}, we get
  \[
    \frac{k_3(\sqrt{3})d}{k_3(\sqrt{3})-1} - o(d) \le \frac{7d}{4} + O_{\alpha,\beta}(1),
  \]
  which implies that $k_3(\sqrt{3}) \ge 7/3$. One can check that the signed graph $H_3^\pm$ in \cref{fig:signed-h3} satisfies $$A_{H_3^\pm}^2 = 3I,$$ and so $\mult(\sqrt{3}, H_3^\pm) = \mult(-\sqrt{3}, H_3^\pm) = 4$. By the Cauchy interlacing theorem, the signed graph $\hat{H}_3^\pm$ in \cref{fig:signed-h3-minus}, which is an induced subgraph of $H_3^\pm$ on $7$ vertices, satisfies $\mult(\sqrt{3}, \hat{H}_3^\pm) = \mult(-\sqrt{3}, \hat{H}_3^\pm) = 3$. Moreover $\chi(\hat{H}_3^\pm) = 3$. Therefore $7/3$ can be achieved for $k_3(\sqrt{3})$ by $\hat{H}_3^\pm$. Now $k_3(\sqrt3) = 7/3$, and \cref{prop:lower-bound} provides a lower bound that matches \cref{eqn:pd-plus-o2} up to an additive constant.
\end{proof}

\begin{figure}
  \centering
    \begin{minipage}[t]{0.5\textwidth}
    \centering
    \begin{tikzpicture}[very thick]
      \draw (0,0) -- (1,0);
      \draw (0,1) -- (0.5,1.5);
      \draw (1.5,0.5) -- (1.5,1.5);
      \draw[dashed] (0,0) -- (0,1) -- (1,1) -- (1,0) -- (1.5,0.5) -- (0.5,0.5) -- cycle;
      \draw[dashed] (0.5,0.5) -- (0.5,1.5) -- (1.5,1.5) -- (1,1);
      \draw[fill] (0,0) circle[radius=.075];
      \draw[fill] (1,0) circle[radius=.075];
      \draw[fill] (0,1) circle[radius=.075];
      \draw[fill] (1,1) circle[radius=.075];
      \draw[fill] (0.5,0.5) circle[radius=.075];
      \draw[fill] (1.5,0.5) circle[radius=.075];
      \draw[fill] (0.5,1.5) circle[radius=.075];
      \draw[fill] (1.5,1.5) circle[radius=.075];
    \end{tikzpicture}
    \captionof{figure}{$H_3^\pm$} \label{fig:signed-h3}
    \end{minipage}%
    \begin{minipage}[t]{0.5\textwidth}
    \centering
    \begin{tikzpicture}[very thick]
      \draw (0,0) -- (1,0);
      \draw (0,1) -- (0.5,1.5);
      \draw[dashed] (0,0) -- (0,1) -- (1,1) -- (1,0) -- (1.5,0.5) -- (0.5,0.5) -- cycle;
      \draw[dashed] (0.5,0.5) -- (0.5,1.5);
      \draw[fill] (0,0) circle[radius=.075];
      \draw[fill] (1,0) circle[radius=.075];
      \draw[fill] (0,1) circle[radius=.075];
      \draw[fill] (1,1) circle[radius=.075];
      \draw[fill] (0.5,0.5) circle[radius=.075];
      \draw[fill] (1.5,0.5) circle[radius=.075];
      \draw[fill] (0.5,1.5) circle[radius=.075];
    \end{tikzpicture}
    \captionof{figure}{$\hat{H}_3^\pm$} \label{fig:signed-h3-minus}
    \end{minipage}
\end{figure}

\section{Algebraic degree argument} \label{sec:algebraic}

We use the following simple observation to derive the asymptotic formula of $N_{\alpha,\beta}(d)$ in the $k_p(\lambda) = \deg(\lambda)$ case, where $\deg(\lambda)$ denotes the algebraic degree of $\lambda$. In particular, the results in this section confirm \cref{conj:main} when $\lambda \in \{\sqrt{2}, \sqrt{3}\}$ and $p \ge \lambda^2 + 1$.

\begin{proposition} \label{prop:kp-at-least-degree}
  For every algebraic integer $\lambda > 0$ and every signed graph $G^\pm$,
  \[
    \mult(\lambda, G^\pm) \le \abs{G^\pm} / \deg(\lambda).
  \]
  In particular, $k_p(\lambda) \ge \deg(\lambda)$ for all $p \in \N$.
\end{proposition}

\begin{proof}
  If $\lambda$ is an eigenvalue of a signed graph $G^\pm$ then each of its  conjugates must also appear with equal multiplicity as eigenvalues of $G^\pm$. Hence $\mult(\lambda, G^\pm)\deg(\lambda) \le \abs{G^\pm}$.
\end{proof}

\begin{proposition} \label{prop:deg-upper-bound}
  For $-1 \le \beta < 0 \le \alpha < 1$, set $\lambda = (1-\alpha)/(\alpha-\beta)$ and $p = \floor{-\alpha/\beta}+1$. If $\lambda$ is an algebraic integer of degree at least $2$, then
  \[
  N_{\alpha,\beta}(d) \le \frac{\deg(\lambda)(d+1)}{\deg(\lambda)-1}.
  \]
  If in addition $k_p(\lambda) = \deg(\lambda)$ and is achievable, then
  \[
    N_{\alpha,\beta}(d) = \frac{\deg(\lambda)d}{\deg(\lambda)-1} + O_{\alpha,\beta}(1).
  \]
\end{proposition}

\begin{proof}
  By \cref{lem:code-graph}, we see that if $G$ is the graph associated to a spherical $\set{\alpha,\beta}$-code of size $N$ in $\R^d$, then, setting $\mu = \alpha/(\alpha - \beta)$ as in \cref{lem:code-graph}, we have
  \[
  d \ge \rank(\lambda I - A_G + \mu J) \ge \rank(\lambda I - A_G) - 1
  = N - \mult(\lambda, G) - 1
  \ge \left(1 - \frac{1}{\deg(\lambda)}\right) N - 1,
  \]
  where the final step applies \cref{prop:kp-at-least-degree}. This yields the first claim.
  If in addition $k_p(\lambda) = \deg(\lambda)$ and is achievable, then \cref{prop:lower-bound} gives a matching lower bound.
\end{proof}

Let us consider the case when $\lambda$ is an algebraic integer of degree $2$. Furthermore suppose that $k_p(\lambda) = 2$ and can be achieved by a signed graph $G^\pm$. Note that both $\lambda$ and its conjugate element $\lambda'$ must have multiplicity $\abs{G^\pm}/2$ as the eigenvalues of $G^\pm$. Because the trace of $A_{G^\pm}$ is $0$, we know that $\lambda + \lambda' = 0$. Therefore $\lambda = \sqrt{n}$ for some $n \in \N$ and $A_{G^\pm}^2 = nI$. It is natural to consider a signed $n$-dimensional hypercube $H^\pm_n$ used by Huang's recent spectacular proof of the sensitivity conjecture \cite[Lemma~2.2]{H19}, in which every square of $H^\pm_n$ contains $1$ or $3$ positive edges.

\begin{figure}
  \centering
  \begin{minipage}[t]{0.5\textwidth}
    \centering
    \begin{tikzpicture}[very thick]
      \draw (0,0) -- (1,0);
      \draw[dashed] (0,0) -- (0,1) -- (1,1) -- (1,0);
      \draw[fill] (0,0) circle[radius=.075];
      \draw[fill] (1,0) circle[radius=.075];
      \draw[fill] (0,1) circle[radius=.075];
      \draw[fill] (1,1) circle[radius=.075];

      \node at (0,-0.25) {};
    \end{tikzpicture}
    \captionof{figure}{$H_2^\pm$} \label{fig:signed-h2}
  \end{minipage}%
  \begin{minipage}[t]{0.5\textwidth}
    \centering
    \begin{tikzpicture}[very thick]
      \coordinate (v) at (2,-1);

      \draw[dashed] (0,0) -- ++(v);
      \draw[dashed] (0,1) -- ++(v);
      \draw[dashed] (1,0) -- ++(v);
      \draw (1,1) -- ++(v);
      \draw (0.5,0.5) -- ++(v);
      \draw[dashed] (1.5,0.5) -- ++(v);
      \draw[dashed] (0.5,1.5) -- ++(v);
      \draw[dashed] (1.5,1.5) -- ++(v);

      \draw (0,0) -- (1,0);
      \draw (0,1) -- (0.5,1.5);
      \draw (1.5,0.5) -- (1.5,1.5);
      \draw[dashed] (0,0) -- (0,1) -- (1,1) -- (1,0) -- (1.5,0.5) -- (0.5,0.5) -- cycle;
      \draw[dashed] (0.5,0.5) -- (0.5,1.5) -- (1.5,1.5) -- (1,1);
      \draw[fill] (0,0) circle[radius=.075];
      \draw[fill] (1,0) circle[radius=.075];
      \draw[fill] (0,1) circle[radius=.075];
      \draw[fill] (1,1) circle[radius=.075];
      \draw[fill] (0.5,0.5) circle[radius=.075];
      \draw[fill] (1.5,0.5) circle[radius=.075];
      \draw[fill] (0.5,1.5) circle[radius=.075];
      \draw[fill] (1.5,1.5) circle[radius=.075];

      \begin{scope}[shift={(v)}]
        \draw (0,0) -- (0,1);
        \draw (1,0) -- (1.5,0.5);
        \draw (1.5,1.5) -- (0.5,1.5);
        \draw[dashed] (0,1) -- (1,1) -- (1,0) -- (0,0) -- (0.5,0.5) -- (0.5,1.5) -- cycle;
        \draw[dashed] (1,1) -- (1.5,1.5) -- (1.5,0.5) -- (0.5,0.5);
        \draw[fill] (0,0) circle[radius=.075];
        \draw[fill] (1,0) circle[radius=.075];
        \draw[fill] (0,1) circle[radius=.075];
        \draw[fill] (1,1) circle[radius=.075];
        \draw[fill] (0.5,0.5) circle[radius=.075];
        \draw[fill] (1.5,0.5) circle[radius=.075];
        \draw[fill] (0.5,1.5) circle[radius=.075];
        \draw[fill] (1.5,1.5) circle[radius=.075];
      \end{scope}
    \end{tikzpicture}
    \captionof{figure}{$H_4^\pm$} \label{fig:signed-h4}
  \end{minipage}
\end{figure}

\begin{proof}[Proof of \cref{thm:main}\ref{thm:main-lambda-2-3}]
  For $\lambda \in \set{\sqrt{2},\sqrt{3}}$ and $p \ge \lambda^2+1$, from \cref{prop:kp-at-least-degree} we know $k_p(\lambda) \ge 2$. In view of \cref{prop:deg-upper-bound} it suffices to prove that $2$ can be achieved for $k_p(\lambda)$. Consider the signed square $H_2^\pm$ in \cref{fig:signed-h2} and the signed cube $H_3^\pm$ in \cref{fig:signed-h3}. In either signed graph, every square contains one positive edge and three negative edges. As a consequence $$A_{H_n^\pm}^2 = nI, \quad\text{for }n=2 \text{ and }3,$$ which implies that the largest eigenvalue of $H_n^\pm$ is $\sqrt{n}$ with multiplicity $2^{n-1}$. It is easy to check that $\chi(H_2^\pm) = 3$ and $\chi(H_3^\pm) = 4$. Thus $k_p(\sqrt{2}) = 2$ for $p \ge 3$, $k_p(\sqrt{3}) = 2$ for $p \ge 4$, and all of them are achievable.
\end{proof}

\begin{remark}
  The constructions $H_2^\pm$ and $H_3^\pm$ in \cref{fig:signed-h3,fig:signed-h2} do not generalize for $\lambda = \sqrt{n}$ with $n \ge 5$ due to the additional constraint on the chromatic number. Suppose that $H^\pm$ is a signed $n$-dimensional hypercube such that $A_{H^\pm}^2 = nI$ and $\chi(H^\pm) < \infty$. Because $A_{H^\pm}^2 = nI$, every square of $H^\pm$ contains odd number of negative edges. Because $\chi(H^\pm) < \infty$, $H^\pm$ has no cycle with exactly one negative edge, and in particular every square of $H^\pm$ contains exactly one positive edge. Unfortunately, this puts a great restriction on $n$. On the one hand, because every positive edge is contained in $n-1$ squares, and each of the $2^{n-2}\binom{n}{2}$ squares in $H^\pm$ contains a positive edge, the number of positive edges is at least $2^{n-2}\binom{n}{2}/(n-1) = n2^{n-3}$. On the other hand, because the positive edges form a matching, there are at most $2^{n-1}$ of them. Therefore $n2^{n-3} \le 2^{n-1}$ and so $n \le 4$. In fact, in addition to $H_2^\pm$ and $H_3^\pm$, the signed $4$-dimensional hypercube $H_4^\pm$ in \cref{fig:signed-h4} satisfies $A_{H_4^\pm}^2 = 4I$ and $\chi(H_4^\pm) = 4$.
\end{remark}

When $k(\lambda) = \deg(\lambda)$, the next result determines $k_p(\lambda)$ for all $p \in \N$. One can then derive the corresponding $N_{\alpha,\beta}(d)$ from \cref{prop:deg-upper-bound}. Note that $k(\lambda) = \deg(\lambda)$ if and only if there exists a graph with spectral radius $\lambda$ whose characteristic polynomial is irreducible. A result of Mowshowitz~\cite{M71} states that such a graph must be asymmetric\footnote{An \emph{asymmetric graph} is a graph for which there are no automorphisms other than the trivial one.}. Asymmetric graphs have at least $6$ vertices. There are $8$ such graphs on $6$ vertices~\cite{ER63}. Among these $8$ asymmetric graphs on $6$ vertices, exactly $7$ of them have irreducible characteristic polynomials,\footnote{It was asserted in \cite[Section~4]{JP20} that all $8$ asymmetric graphs on $6$ vertices have irreducible characteristic polynomials. However the characteristic polynomial of the asymmetric graph
\raisebox{-2pt}{\begin{tikzpicture}[scale=0.2, thick]
  \draw (0,0) -- (6,0);
  \draw (0,0) -- (1,1) -- (2,0) -- (3,1) -- (4,0);
  \draw (1,1) -- (3,1);
  \draw[fill] (0, 0) circle[radius=.1];
  \draw[fill] (1, 1) circle[radius=.1];
  \draw[fill] (2, 0) circle[radius=.1];
  \draw[fill] (3, 1) circle[radius=.1];
  \draw[fill] (4, 0) circle[radius=.1];
  \draw[fill] (6, 0) circle[radius=.1];
\end{tikzpicture}}
is $x(x^5 - 8x^3 - 6x^2 + 8x + 6)$.}
hence their spectral radii satisfy $k(\lambda) = \deg(\lambda)$.

\begin{proposition}
  If $\lambda$ is an algebraic integer and $k(\lambda) = \deg(\lambda)$, then $k_p(\lambda) = \deg(\lambda)$ and is achievable for all $p \in \N$.
\end{proposition}

\begin{proof}
  Clearly $k_p(\lambda) \le k_1(\lambda) = k(\lambda)$. Together with \cref{prop:kp-at-least-degree}, we know that $\deg(\lambda) \le k_p(\lambda) \le k(\lambda)$. Thus if $k(\lambda) = \deg(\lambda)$, then $\deg(\lambda) = k_p(\lambda) = k(\lambda)$, and furthermore $k(\lambda)$ can be achieved for $k_p(\lambda)$ by the smallest graph whose spectral radius is exactly $\lambda$.
\end{proof}

\begin{corollary}
  For $-1 \le \beta < 0 \le \alpha < 1$, set $\lambda = (1-\alpha)/(\alpha-\beta)$ and $p = \floor{-\alpha/\beta}+1$. If $\lambda$ is an algebraic integer and $k(\lambda) = \deg(\lambda)$, then
  \[
    \pushQED{\qed}
    N_{\alpha,\beta}(d) = \frac{\deg(\lambda)d}{\deg(\lambda)-1} + O_{\alpha,\beta}(1). \qedhere
    \popQED
  \]
\end{corollary}

\section{Signed graphs with large eigenvalue multiplicities} \label{sec:large-mult}

In contrast to \cref{thm:kth-ev-mult}, there exist connected signed graphs with bounded maximum degree and chromatic number and linear largest eigenvalue multiplicity. In this section, we show two such constructions. These constructions illustrate an important obstacle to proving \cref{conj:main} following the current framework introduced in \cite{JTYZZ21}.

\begin{example}
  Let $n \geq 3$. Let $G_n^\pm$ be the signed graph consisting of 
  (see \cref{fig:signed-g6} for an illustration of $G_6^\pm$)
  \begin{enumerate}
    \item a positive $n$-cycle on $v_1, v_2, \dots, v_n$,
    \item $n$ copies of a signed $K_5$ with $3$ positive edges forming a $K_3$, and
    \item for each $i \in \set{1, \dots, n}$, a positive edge connecting $v_i$ and $u_i^+$, a negative edge connecting $v_i$ and $u_i^-$, where $u_i^+$ and $u_i^-$ are the two vertices outside the positive $K_3$ in the $i$-th copy of $K_5$.
  \end{enumerate}
\end{example}

\begin{figure}
  \centering
  \begin{minipage}[t]{0.5\textwidth}
    \centering
    \begin{tikzpicture}[scale=0.6, very thick]
      \foreach \r in {45,105,165,225,285,345} {
        \draw[rotate=\r] (1,1) -- (105:{sqrt(2)} );
        \draw[dashed, rotate=\r] (1, 1) -- (1, 2) -- (2, 1) -- (3, 2) -- (1, 2) -- (2, 3) -- (2, 1) -- (3, 3) -- (1, 2);
        \draw[rotate=\r] (3, 2) -- (3, 3) -- (2, 3) -- cycle;
        \draw[rotate=\r] (1, 1) -- (2, 1);
      }
      \foreach \r in {45,105,165,225,285,345} {
        \draw[fill, rotate=\r] (1, 1) circle[radius=.075];
        \draw[fill, rotate=\r] (2, 1) circle[radius=.075];
        \draw[fill, rotate=\r] (1, 2) circle[radius=.075];
        \draw[fill, rotate=\r] (3, 2) circle[radius=.075];
        \draw[fill, rotate=\r] (2, 3) circle[radius=.075];
        \draw[fill, rotate=\r] (3, 3) circle[radius=.075];
      }
    \end{tikzpicture}
    \captionof{figure}{$G_6^\pm$} \label{fig:signed-g6}
  \end{minipage}%
  \begin{minipage}[t]{0.5\textwidth}
    \centering
    \begin{tikzpicture}[scale=0.6, very thick]
      \foreach \r in {45,105,165,225,285,345} {
        \draw[rotate=\r] (1,1) -- (105:{sqrt(2});
        \draw[rotate=\r] (2.5,3.5) -- (2,1) -- (1,1) -- (1,2) -- (3.5,2.5) -- cycle;
        \draw[rotate=\r] (2.5,3.5) -- (2.75,1.75) -- (1.75,2.75) -- (3.5,2.5);
        \draw[rotate=\r] (1.75,2.75) -- (2,1) -- (1,2) -- (2.75,1.75);
      }
      \foreach \r in {45,105,165,225,285,345} {
        \draw[fill, rotate=\r] (1, 1) circle[radius=.075];
        \draw[fill, rotate=\r] (2, 1) circle[radius=.075];
        \draw[fill, rotate=\r] (1, 2) circle[radius=.075];
        \draw[fill, rotate=\r] (1.75, 2.75) circle[radius=.075];
        \draw[fill, rotate=\r] (2.75, 1.75) circle[radius=.075];
        \draw[fill, rotate=\r] (2.5, 3.5) circle[radius=.075];
        \draw[fill, rotate=\r] (3.5, 2.5) circle[radius=.075];
      }
    \end{tikzpicture}
    \captionof{figure}{$H_6$} \label{fig:unsigned-h6}
  \end{minipage}
\end{figure}

So $G_n^\pm$ is a signed graph on $6n$ vertices of maximum degree $5$ and chromatic number $3$. However the multiplicity of its largest eigenvalue is linear in $\abs{G_n^\pm}$. \cref{thm:large-mult-exmaple} is an immediate consequence of the following result.

\begin{proposition}\label{prop:lin-mult}
  The largest eigenvalue of $G_n^\pm$ is $(\sqrt{33} + 1)/2$ with multiplicity $n$.
\end{proposition}

\begin{proof}
  We denote by $K_5^\pm$ the signed $K_5$ with $3$ positive edges forming a $K_3$, and we compute the spectrum of $K_5^\pm$ to be $(1 - \sqrt{33})/2, -1, -1, 1, (1 + \sqrt{33})/2$. Because the largest eigenvalue $(\sqrt{33} + 1)/2$ is simple, by symmetry the corresponding eigenvector assigns the same value to $u_i^+$ and $u_i^-$. For the $i$-th copy of $K_5^\pm$ in $G_n^\pm$, we can extend its top eigenvector to a vector $\bm{x_i}$ on $V(G_n^\pm)$ by padding zeros. Since
  $$
    (A\bm{x_i})_{v_i} = (\bm{x_i})_{u_i^+} - (\bm{x_i})_{u_i^-} = 0,
  $$
  where $A$ denotes the signed adjacency matrix of $G^\pm_n$, the vector $\bm{x_i}$ is also an eigenvector of $G^\pm_n$ associated with the eigenvalue $(\sqrt{33} + 1)/2$.

  For every vector $\bm{x} \in \R^{V(G^\pm_n)}$ that is perpendicular to all $\bm{x_i}$, $1 \le i \le n$, we claim that $\bm{x}^\T A\bm{x} \le 3\bm{x}^\T\bm{x}$, and so all the eigenvalues other than the ones corresponding to $\bm{x_1}, \dots, \bm{x_n}$ are at most $3$. Take such a vector $\bm{x}$, and set $U = \set{u_1^+, u_1^-,\dots, u_n^+,u_n^-}$ and $V = \set{v_1, \dots, v_n}$. We take the orthogonal decomposition $\bm{x} = \bm{y} + \bm{z}$ such that $\bm{y}$ and $\bm{z}$ are supported respectively on $V(G^\pm_n) \setminus V$ and $U \cup V$. In particular, for every $i \in \set{1, \dots, n}$,
  \[
    (\bm{y})_{u_i^+} = (\bm{y})_{u_i^-} = \tfrac{1}{2}\left(\bm{x}_{u_i^+} + \bm{x}_{u_i^-}\right) \quad\text{and}\quad (\bm{z})_{u_i^+} = -(\bm{z})_{u_i^-} = \tfrac{1}{2}\left(\bm{x}_{u_i^+} - \bm{x}_{u_i^-}\right).
  \]
  One can check that $\bm{y}^\T A\bm{z} = 0$. We can simplify
  \[
    \bm{x}^\T A\bm{x} = (\bm{y} + \bm{z})^\T A(\bm{y} + \bm{z}) = \bm{y}^\T A\bm{y} + \bm{z}^\T A \bm{z}.
  \]

  Since $\bm{x}$ and $\bm{z}$ are both orthogonal to each $\bm{x_i}$, so is $\bm{y} = \bm{x}-\bm{z}$. By the Courant--Fischer theorem, we obtain $\bm{y}^\T A\bm{y} \le \lambda_2(K_5^\pm)\bm{y}^\T \bm{y} = \bm{y}^\T \bm{y}$. As $\bm{z}$ is supported on $U \cup V$, we bound $\bm{z}^\T A\bm{z}$ by bounding the spectral radius of $G_n^\pm[U \cup V]$. Since the chromatic number of $G_n^\pm[U \cup V]$ is $2$, the induced signed subgraph shares the same spectral radius with its underlying graph, denoted $H$, on $U \cup V$. Notice that the vector that assigns $1$ to $U$ and $2$ to $V$ is an eigenvector of $H$ with positive components associated with the eigenvalue $3$. By the Perron--Frobenius theorem, the spectral radius of $H$ is~$3$. Thus $\bm{z}^\T A\bm{z} \le 3\bm{z}^\T\bm{z}$. Recall that $\bm{x} = \bm{y} + \bm{z}$ is an orthogonal decomposition. Thus
  \begin{equation*}
    \bm{x}A\bm{x} = \bm{y}^\T A\bm{y} + \bm{z}^\T A \bm{z} \le \bm{y}^\T \bm{y} + 3\bm{z}^\T \bm{z} \le 3(\bm{y}^\T \bm{y} + \bm{z}^\T \bm{z}) = 3\bm{x}^\T\bm{x}. \qedhere
  \end{equation*}
\end{proof}

Even if we restrict the signed graph $G^\pm$ to be all-negative, its largest eigenvalue multiplicity could still be linear in $\abs{G^\pm}$. It suffices to construct the underlying graph $G$ with bounded maximum degree whose smallest eigenvalue multiplicity is linear in $\abs{G}$.

\begin{example}
  Let $n \geq 3$.
  Let $H_n$ be the (unsigned) graph consisting of 
  (see \cref{fig:unsigned-h6} for an illustration of $H_6$)
  \begin{enumerate}
    \item an $n$-cycle on $v_1, v_2, \dots, v_n$,
    \item $n$ copies of $K_{3,3}$, and
    \item for each $i \in \set{1, \dots, n}$, two edges connecting $v_i$ to $u_i^1$ and $u_i^2$, where $u_i^1$ and $u_i^2$ are two adjacent vertices in the $i$-th copy of $K_{3,3}$.
  \end{enumerate}
\end{example}

So $H_n$ is a graph on $7n$ vertices of maximum degree $4$. Moreover, since the chromatic number of $H_n$ is $3$, the corresponding all-negative signed graph has the same chromatic number.

\begin{proposition}
  The smallest eigenvalue of $H_n$ is $-3$ with multiplicity $n$.
\end{proposition}

\begin{proof}
  We compute the spectrum of $K_{3,3}$ to be $3, 0, 0, 0, 0, -3$. For the $i$-th copy of $K_{3,3}$, we can extend the eigenvector associated with its smallest eigenvalue $-3$ to an eigenvector $\bm{x_i}$ on $V(H_n)$ by padding zeros. To prove that all the eigenvalues other than the ones corresponding to $\bm{x_1}, \dots, \bm{x_n}$ are at least $-(1+\sqrt{3})$, it suffices to show that $\bm{x}^\T A\bm{x} \ge -(1+\sqrt{3})\bm{x}^\T\bm{x}$ for every vector $\bm{x} \in \R^{V(G^\pm_n)}$ that is perpendicular to all $\bm{x_i}$, $1 \le i \le n$. Take such a vector $\bm{x}$ and take the orthogonal decomposition $\bm{x} = \bm{y} + \bm{z}$ such that $\bm{y}$ and $\bm{z}$ are supported respectively on $V(H_n)\setminus V$ and $V$, where $V = \set{v_1, \dots, v_n}$. Because $\bm{x}$ and $\bm{z}$ are orthogonal to each $\bm{x_i}$, so is $\bm{y} = \bm{x} - \bm{z}$. By the Courant--Fischer theorem, we obtain $\bm{y}^\T A\bm{y} \ge \lambda_5(K_{3,3})\bm{y}^\T\bm{y} = 0$. We can simplify
  \begin{equation} \label{eqn:gn-estimate}
    \bm{x}^\T A\bm{x} = (\bm{y} + \bm{z})^\T A(\bm{y} + \bm{z}) \ge 2\bm{y}^\T A \bm{z} + \bm{z}^\T A \bm{z},
  \end{equation}
  where $A$ denotes the adjacency matrix of $H_n$. Let $\bar{H}$ be the connected graph consisting of the $n$-cycle on $v_1, \dots, v_n$ and two edges connecting $v_i$ to $u_i^1$ and $u_i^2$ for each $i \in \set{1, \dots, n}$. Let $\bm{\bar{x}}$ be the restriction of $\bm{x}$ on $V(\bar{H})$. Then the right hand side of \cref{eqn:gn-estimate} is equal to $\bm{\bar{x}}^\T\bar{A}\bm{\bar{x}}$, where $\bar{A}$ denotes the adjacency matrix of $\bar{H}$. Notice that the vector that assigns $1+\sqrt{3}$ to $v_i$ and $1$ to both $u_i^1$ and $u_i^2$ for every $i \in \set{1,\dots, n}$ is an eigenvector of $\bar{H}_n$ with positive components associated with the eigenvalue $1+\sqrt{3}$. By the Perron--Frobenius theorem, the spectral radius of $\bar{H}$ is $1+\sqrt{3}$. Thus
  \begin{equation*}
    \bm{x}^\T A\bm{x} \ge \bm{\bar{x}}^\T\bar{A}\bm{\bar{x}} \ge -(1+\sqrt{3})\bm{\bar{x}}^\T\bm{\bar{x}} \ge -(1+\sqrt{3})\bm{x}^\T\bm{x}. \qedhere
  \end{equation*}
\end{proof}

\medskip
\noindent\textbf{Acknowledgments.}
We thank Noga Alon and Colin Defant for discussions and ideas related to the constructions in \cref{sec:large-mult}.
We also thank the anonymous referee for detailed suggestions that significantly improved the exposition of the paper.


\begin{thebibliography}{10}

\bibitem{BDKS18}
Igor Balla, Felix Dr\"{a}xler, Peter Keevash, and Benny Sudakov,
  \emph{Equiangular lines and spherical codes in {E}uclidean space}, Invent.
  Math. \textbf{211} (2018), 179--212.

\bibitem{BY13}
Alexander Barg and Wei-Hsuan Yu, \emph{New bounds for spherical two-distance
  sets}, Exp. Math. \textbf{22} (2013), 187--194.

\bibitem{B16}
Boris Bukh, \emph{Bounds on equiangular lines and on related spherical codes},
  SIAM J. Discrete Math. \textbf{30} (2016), 549--554.

\bibitem{DGS77}
P.~Delsarte, J.~M. Goethals, and J.~J. Seidel, \emph{Spherical codes and
  designs}, Geometriae Dedicata \textbf{6} (1977), 363--388.

\bibitem{ER63}
P.~Erd\H{o}s and A.~R\'{e}nyi, \emph{Asymmetric graphs}, Acta Math. Acad. Sci.
  Hungar. \textbf{14} (1963), 295--315.

\bibitem{GY18}
Alexey Glazyrin and Wei-Hsuan Yu, \emph{Upper bounds for {$s$}-distance sets
  and equiangular lines}, Adv. Math. \textbf{330} (2018), 810--833.

\bibitem{H19}
Hao Huang, \emph{Induced subgraphs of hypercubes and a proof of the sensitivity
  conjecture}, Ann. of Math. \textbf{190} (2019), 949--955.

\bibitem{JP21}
Zilin Jiang and Alexandr Polyanskii, \emph{Forbidden induced subgraphs for
  graphs and signed graphs with eigenvalues bounded from below},
  arXiv:2111.10366.

\bibitem{JP20}
Zilin Jiang and Alexandr Polyanskii, \emph{Forbidden subgraphs for graphs of
  bounded spectral radius, with applications to equiangular lines}, Israel J.
  Math. \textbf{236} (2020), 393--421.

\bibitem{JTYZZ21}
Zilin Jiang, Jonathan Tidor, Yuan Yao, Shengtong Zhang, and Yufei Zhao,
  \emph{Equiangular lines with a fixed angle}, Ann. of Math. \textbf{194}
  (2021), 729--743.

\bibitem{LRS77}
D.~G. Larman, C.~A. Rogers, and J.~J. Seidel, \emph{On two-distance sets in
  {E}uclidean space}, Bull. London Math. Soc. \textbf{9} (1977), 261--267.

\bibitem{LS73}
P.~W.~H. Lemmens and J.~J. Seidel, \emph{Equiangular lines}, J. Algebra
  \textbf{24} (1973), 494--512.

\bibitem{M71}
Abbe Mowshowitz, \emph{Graphs, groups and matrices}, Proceedings of the
  {T}wenty-{F}ifth {S}ummer {M}eeting of the {C}anadian {M}athematical
  {C}ongress ({L}akehead {U}niv., {T}hunder {B}ay, {O}nt., 1971), 1971,
  pp.~509--522.

\bibitem{M09}
Oleg~R. Musin, \emph{Spherical two-distance sets}, J. Combin. Theory Ser. A
  \textbf{116} (2009), 988--995.

\bibitem{N81}
A.~Neumaier, \emph{Distance matrices, dimension, and conference graphs},
  Nederl. Akad. Wetensch. Indag. Math. \textbf{43} (1981), 385--391.

\bibitem{W75}
Richard~M. Wilson, \emph{An existence theory for pairwise balanced designs.
  {III}. {P}roof of the existence conjectures}, J. Combinatorial Theory Ser. A
  \textbf{18} (1975), 71--79.

\bibitem{Y17}
Wei-Hsuan Yu, \emph{New bounds for equiangular lines and spherical two-distance
  sets}, SIAM J. Discrete Math. \textbf{31} (2017), 908--917.

\end{thebibliography}

\end{document}